\pdfoutput=1
\RequirePackage{ifpdf}
\ifpdf 
\documentclass[pdftex]{sigma}
\else
\documentclass{sigma}
\fi

\usepackage{dsfont,mathrsfs}

\numberwithin{equation}{section}

\newtheorem{Theorem}{Theorem}[section]
\newtheorem*{Theorem*}{Theorem}

\newtheorem{Lemma}[Theorem]{Lemma}
\newtheorem{Proposition}[Theorem]{Proposition}
 { \theoremstyle{definition}

\newtheorem*{Example}{Example}
\newtheorem*{Remark}{Remark} }

\makeatletter
\newcommand*{\rom}[1]{\expandafter\@slowromancap\romannumeral #1@}
\makeatother

\DeclareMathOperator{\Sym}{Sym }
\DeclareMathOperator{\Symb}{\overline{Sym} }

\newcommand{\id}{\mathrm{id}}

\DeclareMathOperator{\End}{End}
\newcommand\ltwo{\langle 2\rangle}

\begin{document}
\allowdisplaybreaks

\newcommand{\arXivNumber}{2312.00980}

\renewcommand{\PaperNumber}{060}

\FirstPageHeading

\ShortArticleName{New Combinatorial Formulae for Nested Bethe Vectors}

\ArticleName{New Combinatorial Formulae \\ for Nested Bethe Vectors}

\Author{Maksim KOSMAKOV~$^{\rm a}$ and Vitaly TARASOV~$^{\rm b}$}
\AuthorNameForHeading{M.~Kosmakov and V.~Tarasov}

\Address{$^{\rm a)}$~Department of Mathematical Sciences, University of Cincinnati,\\
\hphantom{$^{\rm a)}$}~P.O.~Box 210025, Cincinnati, OH 45221, USA}
\EmailD{\href{mailto:kosmakmm@ucmail.uc.edu}{kosmakmm@ucmail.uc.edu}}

\Address{$^{\rm b)}$~Department of Mathematical Sciences, Indiana University Indianapolis,\\
\hphantom{$^{\rm b)}$}~402 North Blackford St, Indianapolis, IN 46202-3216, USA}
\EmailD{\href{mailto:vtarasov@iu.edu}{vtarasov@iu.edu}}

\ArticleDates{Received January 08, 2025, in final form July 08, 2025; Published online July 22, 2025}

\Abstract{We give new combinatorial formulae for vector-valued weight functions (off-shell nested Bethe vectors) for the evaluation modules over the Yangian $Y(\mathfrak{gl}_4)$. The case of~$Y(\mathfrak{gl}_n)$ for an arbitrary $n$ is considered in [\textit{Lett. Math. Phys.} \textbf{115} (2025), 12, 20~pages, arXiv:2402.15717].}


\Classification{17B37; 81R50; 82B23}

\section{Introduction}
In this paper, we will give new combinatorial formulae for vector-valued weight functions for
evaluation modules over the Yangian $Y(\mathfrak{gl}_n)$. The weight functions, also known as
(off-shell) nested Bethe vectors, play an important role in the theory of quantum integrable
models and representation theory of Lie algebras and quantum groups. Initially, they appeared
in the framework of the nested algebraic Bethe ansatz as a tool to find eigenvectors and
eigenvalues of transfer matrices of lattice integrable models associated with higher rank Lie
algebras \cite{KulRes83, KulRes82}, see \cite{S1, S2} for a review of the algebraic Bethe ansatz.
The results of \cite{KulRes83} have been extended to higher transfer matrices in \cite{MTV1}.

Furthermore, the vector-valued weight functions were used to construct hypergeometric solutions
of the quantized (difference) Knizhnik--Zamolodchikov equations \cite{MTT, TV}. They also showed
up in several related problems \cite{KPT, MTV2, TV3, TV4}. In a more recent development,
the weight functions were connected to the stable envelopes for particular Nakajima quiver
varieties, the cotangent bundles of partial flag varieties \cite{RTV1, RTV2, RTV3, RTV4, TV5}.

For various applications, it is important to have expressions for vector-valued weight functions
for tensor products of evaluation modules over $Y(\mathfrak{gl}_n)$. Such expressions can be
obtained in two steps. The first step is to consider weight functions for a single evaluation
module, and the second step is to combine expressions for individual evaluation modules
into an expression for the whole tensor product. In this paper, we will focus on the first step.
The second step is fairly standard and is not specifically discussed here.

By definition, an evaluation $Y(\mathfrak{gl}_n)$-module is a $\mathfrak{gl}_n$-module
equipped with the action of~$Y(\mathfrak{gl}_n)$ via the evaluation homomorphism
$Y(\mathfrak{gl}_n)\rightarrow U(\mathfrak{gl}_n)$, see Section~\ref{notation}. The goal is
to expand the vector-valued weight function for the evaluation $Y(\mathfrak{gl}_n)$-module
in a basis coming from the representation theory of $\mathfrak{gl}_n$ and find expressions
for the coordinates. For Verma modules over~$\mathfrak{gl}_n$, expressions of this kind are given
in~\cite{TVC}. In this paper, we give a generalization of formulae from~\cite{TVC}.

Combinatorial formulae for the vector-valued weight functions associated with the differential
Knizhnik--Zamolodchikov equations were developed in \cite{FRV,MV, M,RSV,SV1,SV2}.

The expressions for weight functions in \cite{TVC} are based on recursions
induced by the standard embeddings of Lie algebras,
$\mathfrak{gl}_1\oplus \mathfrak{gl}_{n-1}\subset \mathfrak{gl}_n$ and
$\mathfrak{gl}_{n-1}\oplus \mathfrak{gl}_1 \subset \mathfrak{gl}_n$. The recursions
allow one to write down weight functions for $Y(\mathfrak{gl}_n)$ via weight functions for
$Y(\mathfrak{gl}_{n-1})$. This results in~formulae for coordinates of weight functions in bases
of Verma $\mathfrak{gl}_n$-modules of the form
\begin{equation}\label{basis}
\biggl\{ \prod_{i>j} e_{ij}^{m_{ij}}v, \,
m_{ij}\in\mathbb{Z}_{\geq0}\biggr\},
\end{equation}
where $e_{ij}$ are the standard generators of $\mathfrak{gl}_n$, see~\eqref{eijkl},
$v$ is the highest weight
vector, and some ordering of noncommuting factors is imposed. The ordering is determined
by the in-between part of the involved chain of embeddings
$\mathfrak{gl}_1 \oplus\dots\oplus \mathfrak{gl}_1\subset\dots\subset\mathfrak{gl}_n$.
For instance, the chain
\[\mathfrak{gl}_1 \oplus\dots\oplus \mathfrak{gl}_1\subset
\dots\subset\mathfrak{gl}_{n-2}\oplus \mathfrak{gl}_1 \oplus \mathfrak{gl}_1 \subset
\mathfrak{gl}_{n-1}\oplus \mathfrak{gl}_1 \subset \mathfrak{gl}_n
\]
yields the ordering
\begin{gather}\label{eq1-add}
\text{$e_{ij}^\circledast$ is to the left of $e_{kl}^\circledast$ if $i>k$ or $i=k$, $j>l$,}
\end{gather}
while the chain
\[
\mathfrak{gl}_1 \oplus\dots\oplus \mathfrak{gl}_1
\subset\dots\subset\mathfrak{gl}_1 \oplus \mathfrak{gl}_1\oplus\mathfrak{gl}_{n-2} \subset
\mathfrak{gl}_1\oplus \mathfrak{gl}_{n-1}\subset \mathfrak{gl}_n
\]
yields the ordering
\begin{gather}\label{eq2-add}
\text{$e_{ij}^\circledast$ is to the left of $e_{kl}^\circledast$ if $j<l$ or $j=l$, $i<k$.}
\end{gather}
For example, for $n=4$, the product $e_{43}e_{42}e_{41}e_{32}e_{31}e_{21}$ obeys ordering~\eqref{eq1-add}, while the product $e_{21}e_{31}e_{41}e_{32}e_{42}e_{43}$ obeys ordering~\eqref{eq2-add}.

However, some natural orderings of noncommuting factors in~\eqref{basis} important
for applications do not show up in the formulae established in \cite{TVC},
see, for instance, \cite{MV}.
The first nontrivial example occurs at $n=4$ and is given by the basis
\begin{equation}\label{basis4}
\bigl\{e_{32}^{m_{32}}e_{31}^{m_{31}}e_{42}^{m_{42}}e_{41}^{m_{41}}e_{21}^{m_{21}}e_{43}^{m_{43}}v,\,
 m_{ij}\in\mathbb{Z}_{\geq0}\bigr\}.
\end{equation}
To make the set of covered orderings wider, one can consider recursions based on more general
embeddings
\begin{equation}\label{newembed}
\mathfrak{gl}_m\oplus\mathfrak{gl}_{n-m}\subset \mathfrak{gl}_n \qquad \text{with}\ 1<m<n-1.
\end{equation}
For instance, the embedding
$\mathfrak{gl}_2\oplus \mathfrak{gl}_2\subset \mathfrak{gl}_4$ yields example~\eqref{basis4}.
In this paper, we will work out example~\eqref{basis4} in detail with the main result
given by Theorem~\ref{main2}. We consider the general case in \cite{KT}.

We would like to present the $\mathfrak{gl}_4$ case separately in order to explain calculations
more clearly without introducing too cumbersome notation and to make the exposition more
accessible. For the same purpose, we show explicitly intermediate steps in the proofs that
might commonly be omitted for the sake of making a paper shorter. In particular, we give in
Appendix~\ref{appendixA} a proof of Proposition~\ref{twotensor}. Although this statement has a long history,
going back to \cite{Kor}, numerous applications, and is explained in several lecture courses,
see~\cite{S2}, its straightforward proof is not easily available in the literature.

At the same time we point out that the proof of Theorem~\ref{main2} in this paper extends
almost in a straightforward way to the proof of \cite[Theorem~5.5]{KT} in the general
$\mathfrak{gl}_n$ case. In particular, the proof of the key Proposition~6.3 in \cite{KT}
is literally the same as the proof of Proposition~\ref{symprop} in~this~paper.

Unlike \cite{TVC}, we will consider only the case of weight functions for Yangian modules
(the rational case). It turns out that dealing with weight functions for modules over
the quantum loop algebra \smash{$U_q\bigl(\widetilde{\mathfrak{gl}_n}\bigr)$}, the trigonometric case,
hits an obstacle of essential noncommutativity of $q$-analogues of the generators
$e_{ij}$, $i>j$. This obstacle does not show up for the embeddings
$\mathfrak{gl}_1\oplus \mathfrak{gl}_{n-1}\subset \mathfrak{gl}_n$ and
$\mathfrak{gl}_{n-1}\oplus \mathfrak{gl}_1 \subset \mathfrak{gl}_n$ explored in \cite{TVC},
but reveals itself for embeddings~\eqref{newembed}. For instance, the obstacle
in example~\eqref{basis4} comes from the relation
\begin{equation*}
e_{42}e_{31}-e_{31}e_{42}=\bigl(q-q^{-1}\bigr)e_{32}e_{41}
\end{equation*}
that holds in the trigonometric case.

There is an alternative approach to get explicit expressions for the vector-valued weight
functions in the trigonometric case, see \cite{KP,KPT, OPS}, based on considering composed
currents and half-currents in the quantum affine algebra and their projections on two Borel
subalgebras of different kinds. This approach allows one to recover the combinatorial expressions
for vector-valued weight functions in evaluation modules in the trigonometric case obtained
in \cite{TVC}. It is an~interesting open question whether the composed currents approach
can be helpful to obtain trigonometric analogues of new combinatorial expressions for
vector-valued weight functions developed in this paper.

\section{Notations}\label{notation}
We will be using the standard superscript notation for embeddings of tensor factors into tensor products. For a tensor product of vector spaces $V_1\otimes V_2\otimes\dots\otimes V_k$ and an operator $A\in \End (V_i)$, denote
\begin{equation*}
A^{(i)}=1^{\otimes(i-1)} \otimes A \otimes 1^{\otimes(k-i)} \in \End(V_1\otimes V_2\otimes\dots\otimes V_k). \end{equation*}
Also, if $B\in \End (V_j)$, $i\neq j$, denote $(A\otimes B)^{(ij)}=A^{(i)}B^{(j)}$, etc.

Fix a positive integer $n$. Throughout the paper, we identify elements of End $\mathbb{C}^n$ with $n\times n$ matrices using the standard basis of $\mathbb{C}^n$. That is, for $L\in \End \mathbb{C}^n$ we have $ L= \bigl( L_b^a \bigr)_{a,b=1}^n$, where~$L_b^a$ are the entries of $L$. Entries of matrices acting in the tensor products \smash{$ (\mathbb{C}^n )^{\otimes k}$} are naturally labeled by multiindices. For instance, if \smash{$M\in \End(\mathbb{C}^n\otimes \mathbb{C}^n)$}, then \smash{$M= \bigl( M_{cd}^{ab} \bigr)_{a,b,c,d=1}^n$}.

The rational $R$-matrix is $R(u)\in \End (\mathbb{C}^n\otimes \mathbb{C}^n)$,
\begin{equation}\label{rmatrix}
R(u)=1+\frac{1}{u}\sum_{a, b=1}^{n} E_{a b} \otimes E_{b a},
\end{equation}
where $ E_{a b} \in \End\left(\mathbb{C}^{n}\right) $ is the matrix with the only nonzero entry equal to 1 at the intersection of the $a$-th row and $b$-th column. The entries of $R(u)$ are
\begin{equation*}
R^{ab}_{cd}(u)= \delta_{ac}\delta_{bd} +\frac{1}{u}\delta_{ad}\delta_{bc}.
\end{equation*}
The $R$-matrix satisfies the Yang--Baxter equation
\begin{equation}\label{YB} R^{(12)}(u-v) R^{(13)}(u) R^{(23)}(v)=R^{(23)}(v) R^{(13)}(u) R^{(12)}(u-v).\end{equation}

The Yangian $Y(\mathfrak{g l}_{n})$ is an unital associative algebra with generators \smash{$ \bigl(T^a_b\bigr)^{\{s\}}$}, $a, b=1, \dots, n, $ and $s=1,2, \dots$. Organize them into generating series
\begin{equation}\label{toperator} T^a_b(u)=\delta_{ab}+\sum_{s=1}^{\infty} \bigl(T^a_b\bigr)^{\{s\}} u^{-s}, \qquad a, b=1, \dots, n.\end{equation}
The defining relations in $Y(\mathfrak{g l}_{n})$ are
\begin{equation} (u-v)\bigl[T^a_b(u), T^c_d(v)\bigr]=T^a_d(u)T^c_b(v) -T^a_d(v)T^c_b(u) \label{trel}\end{equation}
for all $a, b, c, d=1, \dots, n$.

Combine series $\eqref{toperator}$ into a matrix $T(u)=\sum_{a, b=1}^{n} E_{a b} \otimes T_{ b}^a(u)$ with entries in $ Y(\mathfrak{g l}_{n})$. Then relations~\eqref{trel} amount to the following equality:
\begin{equation*}
R^{(12)}(u-v) T^{(1)}(u) T^{(2)}(v)=T^{(2)}(v) T^{(1)}(u) R^{(12)}(u-v),
\end{equation*}
where $T^{(1)}(u)=\sum_{a, b=1}^{n} E_{a b}\otimes 1 \otimes T_{ b}^a(u)$ and $T^{(2)}(v)=\sum_{a, b=1}^{n} 1 \otimes E_{a b}\otimes T_{ b}^a(v)$.

The Yangian $Y(\mathfrak{g l}_{n})$ is a Hopf algebra. In terms of generating series $\eqref{toperator}$, the coproduct $\Delta\colon Y(\mathfrak{g l}_{n}) \rightarrow Y(\mathfrak{g l}_{n}) \otimes Y(\mathfrak{g l}_{n})$ reads as follows:
\begin{equation}\label{coproduct}
\Delta\bigl(T_{ b}^a(u)\bigr)=\sum_{c=1}^{n} T^c_{ b}(u) \otimes T^a_{ c}(u), \qquad a, b=1, \dots, n.
\end{equation}
Denote by $\widetilde{\Delta}\colon Y(\mathfrak{g l}_{n}) \rightarrow Y(\mathfrak{g l}_{n}) \otimes Y(\mathfrak{g l}_{n})$ the opposite coproduct
\begin{equation}\label{opcoproduct}
\widetilde{\Delta}\bigl(T_{ b}^a(u)\bigr)=\sum_{c=1}^nT^a_c(u)\otimes T^c_b(u), \qquad a, b=1, \dots, n.
\end{equation}
There is a one-parameter family of automorphisms $\rho_x\colon Y(\mathfrak{gl}_n) \rightarrow Y(\mathfrak{gl}_n)$ defined in terms of the series $T(u)$ by the rule $\rho_x T(u)= T(u-x)$, where, on the right-hand side, each expression $(u-x)^{-s}$ has to be expanded as a power series in $u^{-1}$.

Denote by $e_{ab}$, $a,b=1,\dots,n$, the standard generators of the Lie algebra $\mathfrak{gl}_n$,
\begin{equation}\label{eijkl}
[e_{ab},e_{cd}]=e_{ad} \delta_{bc}-e_{cb} \delta_{ad}.
\end{equation}
A vector $v$ in a $\mathfrak{gl}_n$-module is called
singular of weight $\bigl(\Lambda^1,\dots,\Lambda^n\bigr)$ if $e_{ab}v = 0$ for all $a<b$ and~${e_{aa}v=\Lambda^av}$ for all $a=1,\dots,n$.

The Yangian $Y(\mathfrak{g l}_{n})$ contains the universal enveloping algebra $U(\mathfrak{g l}_{n})$ as a Hopf subalgebra. The embedding is given by the rule \smash{$e_{a b} \mapsto \bigl(T_{ a}^b\bigr)^{\{1\}}$} for all $a, b=1, \dots, n$. We identify $U(\mathfrak{g l}_{n})$ with its image in $Y(\mathfrak{g l}_{n})$ under this embedding.

The evaluation homomorphism $\epsilon\colon Y (\mathfrak{gl}_n) \rightarrow U(\mathfrak{gl}_n)$ is given by the rule $\epsilon\colon T_{b}^a(u) \mapsto \delta_{ab}+e_{ba}u^{-1}$ for all $a,b = 1,\dots,n$. Both the automorphisms $\rho_x$ and the homomorphism $\epsilon$ restricted to the subalgebra $U(\mathfrak{gl}_n)$ are the identity maps.

For a $\mathfrak{gl}_n$-module $V$, denote by $V(x)$ the $Y(\mathfrak{gl}_n)$-module induced from $V$ by the homomorphism~${\epsilon\circ \rho_x}$. The module $V(x)$ is called an evaluation module over $Y(\mathfrak{gl}_n)$.

A vector $v$ in a $Y(\mathfrak{gl}_n)$-module is called singular with respect to the action of $Y(\mathfrak{gl}_n)$ if ${T^{a}_{b}(u)v = 0}$ for all $1\leq b < a \leq n$. A singular vector $v$ that is an eigenvector for the action of $T_{1}^{1}(u),\dots,T_{n}^{n}(u)$ is called a weight singular vector, and the respective eigenvalues are denoted by $\bigl\langle T^1_{1}(u)v\bigr\rangle,\dots, \langle T^n_{n}(u)v\rangle$.

\begin{Example}
Let $V$ be a $\mathfrak{g l}_{n}$-module and $v \in V$ be a $\mathfrak{gl}_n$-singular vector of weight $\bigl(\Lambda^{1}, \dots, \Lambda^{n}\bigr)$. Then $v$ is a weight singular vector with respect to the action of $Y(\mathfrak{g l}_{n})$ in the evaluation module~$V(x)$ and $\langle T^a_{ a}(u) v\rangle=1+\Lambda^{a}(u-x)^{-1}$, $a=1, \dots, n$.
\end{Example}

For $k<n$, we consider two embeddings of the algebra $Y(\mathfrak{gl}_{k})$ into $Y(\mathfrak{gl}_{n})$, called $\phi_k$ and $\psi_k$:
\begin{gather}\label{embeddings}
\phi_k\bigl(T^{\langle k\rangle}(u)\bigr)^a_b=\bigl(T^{\langle n\rangle}(u)\bigr)^a_b,\qquad
\psi_k\bigl(T^{\langle k\rangle}(u)\bigr)^a_b=\bigl(T^{\langle n\rangle}(u)\bigr)^{a+n-k}_{b+n-k},
\end{gather}
with $a,b=1,\dots, k$. Here $\bigl(T^{\langle k\rangle}(u)\bigr)^a_b$ and $\bigl(T^{\langle n\rangle}(u)\bigr)^a_b$ are series $T^a_b(u)$ for the algebras $Y(\mathfrak{gl}_{k})$ and $Y(\mathfrak{gl}_{n})$, respectively.

\section{Combinatorial formulae for rational weight functions}

Fix a collection of nonnegative integers $\xi_1,\xi_2,\dots,\xi_{n-1}$. Set $\boldsymbol{\xi}=(\xi_1,\xi_2,\dots,\xi_{n-1})$ and $\xi^a=\xi_1+\dots+\xi_a$, $a=1,\dots,n-1$.
Consider the variables $t^a_i$, $a=1,\dots, n-1$, $i=1,\dots, \xi_a$. We will also write
\begin{equation*}
\boldsymbol{t}^a=\bigl(t^a_1,\dots, t^a_{\xi_a}\bigr),\qquad \boldsymbol{t}=\bigl(\boldsymbol{t}^1,\dots,\boldsymbol{t}^{n-1}\bigr).
\end{equation*}
We will use the ordered product notation for any noncommuting factors $X_1,\dots,X_k$,
\begin{equation*}
\overset{\rightarrow}{\prod\limits_{1\leq i\leq k}} X_i= X_1X_2\cdots X_k, \qquad \overset{\leftarrow}{\prod\limits_{1\leq i\leq k}} X_i= X_kX_{k-1}\cdots X_1.
\end{equation*}
Consider the vector space \smash{$(\mathbb{C}^n)^{\otimes \xi^{n-1}}$} and define
\begin{gather*}
\overset{[j]}{\mathbb{T}}\bigl({\boldsymbol{t}^j}\bigr)=\overset{\rightarrow}{\prod\limits_{ 1\leq k\leq \xi_j}} {T}^{(\xi^{j-1}+k)}\bigl(t_k^j\bigr),\\ \overset{[k,j]}{\mathbb{R}}\bigl(\boldsymbol{t}^k,\boldsymbol{t}^j\bigr)=
\overset{\rightarrow}{\prod\limits_{1\leq i\leq\xi_k}}\Biggl(
\overset{\leftarrow}{\prod\limits_{1\leq l\leq\xi_j}} {R}^{(\xi^{k-1}+i,\xi^{j-1}+l)}\bigl(t^k_i-t^j_l\bigr)\Biggr),
\end{gather*}
where we consider \smash{${T}^{(\xi^{j-1}+k)}\bigl(t_k^j\bigr)$} as a matrix with noncommuting entries belonging to $Y(\mathfrak{gl}_{n})$.

For the expression
\begin{equation}\label{betthe 1}
\widehat{\mathbb{T}}_{\boldsymbol{\xi}}(\boldsymbol{t})= \overset{[1]}{\mathbb{T}}\bigl(\boldsymbol{t}^1\bigr)\cdots \overset{[n-1]}{\mathbb{T}}\bigl(\boldsymbol{t}^{n-1}\bigr)\overset{\leftarrow}{\prod\limits_{1\leq i\leq n-1}}\Biggl(\overset{\leftarrow}{\prod\limits_{1\leq j< i}}\overset{[i,j]}{\mathbb{R}}\bigl(\boldsymbol{t}^i,\boldsymbol{t}^j\bigr)\Biggr),
\end{equation}
denote by ${\mathbb{B}}_{\boldsymbol{\xi}}(\boldsymbol{t})$ the following entry
\begin{equation*}
{\mathbb{B}}_{\boldsymbol{\xi}}(\boldsymbol{t})=\bigl( \widehat{\mathbb{T}}_{\boldsymbol{\xi}}(\boldsymbol{t}) \bigr)^{\boldsymbol{1}^{\xi_1}, \boldsymbol{2}^{\xi_2}, \dots, \boldsymbol{n-1}^{\xi_{n-1}}}_{\boldsymbol{2}^{\xi_1}, \boldsymbol{3}^{\xi_2}, \dots, \boldsymbol{n}^{\xi_{n-1}}},
\end{equation*}
where
\begin{align*}
&\boldsymbol{1}^{\xi_1}, \boldsymbol{2}^{\xi_2}, \dots, \boldsymbol{(n-1)}^{\xi_{n-1}} =\underbrace{ 1, 1,\dots, 1}_{\xi_1}, \underbrace{ 2,2,\dots,2}_{\xi_2}, \dots,\underbrace{n-1, n-1, \dots, n-1}_{\xi_{n-1}},\\
&\boldsymbol{2}^{\xi_1}, \boldsymbol{3}^{\xi_2}, \dots, \boldsymbol{n}^{\xi_{n-1}} =\underbrace{ 2, 2,\dots, 2}_{\xi_1}, \underbrace{ 3,3,\dots,3}_{\xi_2}, \dots, \underbrace{n, n, \dots, n}_{\xi_{n-1}}.
\end{align*}
To indicate the dependence on $n$, if necessary, we will write \smash{${\mathbb{B}}_{\boldsymbol{\xi}}^{\langle n\rangle}(\boldsymbol{t})$}.
\begin{Example}
Let $n=2$ and $\boldsymbol{\xi}=(\xi_{1}) $. Then ${\mathbb{B}}_{\boldsymbol{\xi}}^{\ltwo}(\boldsymbol{t})=T^1_{2} \bigl(t_{1}^{1} \bigr) \dots T^1_{2} \bigl(t_{\xi_1^{}}^{1}\bigr)$.
Abusing notation, we will further write \smash{${\mathbb{B}}_{\xi_1}^{\ltwo}(\boldsymbol{t})$}
instead of \smash{${\mathbb{B}}_{\boldsymbol{\xi}}^{\ltwo}(\boldsymbol{t})$}.
\end{Example}
\begin{Example} Let $n=4$ and $\boldsymbol{\xi}=(1,1,1)$. Then
\begin{align*}
{\mathbb{B}}_{\boldsymbol{\xi}}^{\langle 4\rangle}(\boldsymbol{t})={}& T^1_{2}\bigl(t_{1}^{1}\bigr) T^2_{3}\bigl(t_{1}^{2}\bigr) T^3_{4}\bigl(t_{1}^{3}\bigr) \\
&{+}\frac{1}{t_{1}^{2}-t_{1}^{1}} T^1_{3}\bigl(t_{1}^{1}\bigr) T^2_{2}\bigl(t_{1}^{2}\bigr) T^3_{4}\bigl(t_{1}^{3}\bigr)+\frac{1}{t_{1}^{3}-t_{1}^{2}} T^1_{2}\bigl(t_{1}^{1}\bigr) T^2_{4}\bigl(t_{1}^{2}\bigr) T^3_{3}\bigl(t_{1}^{3}\bigr) \\
&{+}\frac{1}{\bigl(t_{1}^{2}-t_{1}^{1}\bigr)\bigl(t_{1}^{3}-t_{1}^{2}\bigr)}\bigl(T^1_{4}\bigl(t_{1}^{1}\bigr) T^2_{2}\bigl(t_{1}^{2}\bigr) T^3_{3}\bigl(t_{1}^{3}\bigr)+T^1_{3}\bigl(t_{1}^{1}\bigr) T^2_{4}\bigl(t_{1}^{2}\bigr) T^3_{2}\bigl(t_{1}^{3}\bigr)\bigr) \\
&{+}\frac{\bigl(t_{1}^{2}-t_{1}^{1}\bigr)\bigl(t_{1}^{3}-t_{1}^{2}\bigr)+1}{\bigl(t_{1}^{2}-t_{1}^{1}\bigr)\bigl(t_{1}^{3}-t_{1}^{1}\bigr)\bigl(t_{1}^{3}-t_{1}^{2}\bigr)} T^1_{4}\bigl(t_{1}^{1}\bigr) T^2_{3}\bigl(t_{1}^{2}\bigr) T^3_{2}\bigl(t_{1}^{3}\bigr).
\end{align*}\end{Example}

For a weight singular vector $v$ with respect to the action of $Y(\mathfrak{gl}_n)$, we call the expression $\mathbb{B}_{\boldsymbol{\xi}}(\boldsymbol{t})v$ the (rational) vector-valued weight function of weight ($\xi_1,\xi_2-\xi_1,\dots,\xi_{n-1}-\xi_{n-2},-\xi_{n-1}$) associated with $v$.

From now on, we will consider only the case $n=4$. We are interested in writing down the following expansion for a weight function in an evaluation module over the $Y(\mathfrak{gl}_4)$:
\begin{equation}\label{example}
{\mathbb{B}}_{\boldsymbol{\xi}}(\boldsymbol{t})v=\sum\limits_{\vec{m}\in\mathbb{Z}_{\geq 0}^6} F_{\vec{m}}(\boldsymbol{t})\cdot e^{m_{32}}_{32}e^{m_{31}}_{31}e^{m_{42}}_{42}e^{m_{41}}_{41}e^{m_{21}}_{21}e^{m_{43}}_{43} v
\end{equation}
with the functions $F_{\vec{m}}(\boldsymbol{t})$ given by explicit formulae. Various
similar expansions for $ {\mathbb{B}}_{\boldsymbol{\xi}}(\boldsymbol{t})v$ were obtained in \cite{TVC}, however, expansion $\eqref{example}$ is not covered there.

\section{Splitting property of the weight functions}

Let \smash{$T_{a b}^{\ltwo}(u)$} be series ($\ref{toperator}$) for the algebra $Y(\mathfrak{g l}_{2})$, and \smash{$R^{\ltwo}(u)$} be the corresponding rational $R$-matrix, see $\eqref{rmatrix}$. Consider two $Y(\mathfrak{g l}_{2})$-module structures on the vector space $\mathbb{C}^2$. The first one, called $L(x)$, is given by the rule
\begin{equation*}
\pi(x)\colon \ T^{\ltwo}(u) \mapsto R^{\ltwo}(u-x),
\end{equation*}
and the second one, called $\bar{L}(x)$, is given by the rule
\begin{equation*}
\varpi(x)\colon \ T^{\ltwo}(u) \mapsto\bigl( \bigl(R^{\ltwo}(x-u)\bigr)^{(21)}\bigr)^{t_{2}},
\end{equation*}
where the superscript $t_{2}$ stands for the matrix transposition in the second tensor factor.

Let $\mathbf{w}_{1}$, $\mathbf{w}_{2}$ be the standard basis of the space $\mathbb{C}^{2}$. The module $L(x)$ is a highest weight evaluation module with $\mathfrak{g l}_{2} $ highest weight $(1, 0)$ and highest weight vector~$\mathbf{w}_{1}$. The module~$\bar{L}(x)$ is a~highest weight evaluation module with $\mathfrak{g l}_{2}$ highest weight $(0,-1)$ and highest weight vector~$\mathbf{w}_{2}$.
For any $X\in \operatorname{End}\bigl(\mathbb{C}^{2}\bigr)$, set $\nu(X)=X \mathbf{w}_{1}$ and $\bar{\nu}(X)=X \mathbf{w}_{2}$.

Recall the coproducts $\Delta$ and \smash{$\widetilde{\Delta}$}, see~\eqref{coproduct} and~\eqref{opcoproduct}, and the embeddings $\psi_2\colon Y(\mathfrak{g l}_{2})\rightarrow Y(\mathfrak{g l}_{4})$ and $\phi_2\colon Y(\mathfrak{g l}_{2})\rightarrow Y(\mathfrak{g l}_{4})$ given by $\eqref{embeddings}$.
For any $k$, denote by
\[
 \Delta^{(k)}\colon \ Y(\mathfrak{g l}_{2}) \rightarrow(Y(\mathfrak{g l}_{2}))^{\otimes(k+1)} \qquad \text{and}\qquad \widetilde{\Delta^{(k)}}\colon \ Y(\mathfrak{g l}_{2}) \rightarrow(Y(\mathfrak{g l}_{2}))^{\otimes(k+1)}
 \]
 the iterated coproduct and opposite coproduct.
Consider the maps
\begin{gather*}\psi_2(x_{1}, \dots, x_{k})\colon\ Y(\mathfrak{g l}_{2}) \rightarrow\bigl(\mathbb{C}^{2}\bigr)^{\otimes k} \otimes Y(\mathfrak{g l}_{4}),
\\
\psi_2(x_{1}, \dots, x_{k})= \bigl(\nu^{\otimes k} \otimes \mathrm{id}\bigr) \circ (\pi (x_{1} ) \otimes \cdots \otimes \pi (x_{k} ) \otimes \psi_2 ) \circ\Delta^{(k)},
\end{gather*}
and
\begin{gather*}
\begin{split}
& \phi_2 (x_{1}, \dots, x_{k} )\colon\ Y(\mathfrak{g l}_{2}) \rightarrow\bigl(\mathbb{C}^{2}\bigr)^{\otimes k} \otimes Y(\mathfrak{g l}_{4}),
\\
& \phi_2 (x_{1}, \dots, x_{k} )= \bigl(\bar{\nu}^{\otimes k} \otimes \mathrm{id} \bigr) \circ ( \varpi (x_{1} ) \otimes \cdots \otimes \varpi (x_{k} ) \otimes \phi_2 ) \circ \widetilde{\Delta} ^{(k)}.
\end{split}
\end{gather*}
For any element \smash{$g \in\bigl(\mathbb{C}^{2}\bigr)^{\otimes k} \otimes Y(\mathfrak{g l}_{4})$}, we define its components $g^{\boldsymbol{a}}$, $\boldsymbol{a}=(a_{1}, \dots, a_{k})$, by the rule
\begin{equation*}
g=\sum_{a_{1}, \dots, a_{k}=1}^{2} \mathbf{w}_{a_{1}} \otimes \cdots \otimes \mathbf{w}_{a_{k}} \otimes g^{\boldsymbol{a}}.
\end{equation*}

In the $\mathfrak{gl}_4$ case, we have $\boldsymbol{\xi}=(\xi_1,\xi_2,\xi_3)$, and formula $\eqref{betthe 1}$ takes the form
\begin{eqnarray*}
{\mathbb{B}}_{\boldsymbol{\xi}}(\boldsymbol{t})=\Bigl(\overset{[1]}{\mathbb{T}}\bigl(\boldsymbol{t}^1\bigr)\overset{[2]}{\mathbb{T}}\bigl(\boldsymbol{t}^2\bigr)\overset{[3]}{\mathbb{T}}\bigl(\boldsymbol{t}^3\bigr)
\overset{[32]}{\mathbb{R}}\bigl(\boldsymbol{t}^3,\boldsymbol{t}^2\bigr)\overset{[31]}{\mathbb{R}}\bigl(\boldsymbol{t}^3,\boldsymbol{t}^1\bigr)\overset{[21]}{\mathbb{R}}\bigl(\boldsymbol{t}^2,\boldsymbol{t}^1\bigr)
\Bigr)^{\boldsymbol{1}^{\xi_1}, \boldsymbol{2}^{\xi_2}, \boldsymbol{3}^{\xi_3}}_{\boldsymbol{2}^{\xi_1}, \boldsymbol{3}^{\xi_2}, \boldsymbol{4}^{\xi_3}}.
\end{eqnarray*}

\begin{Proposition}Let $v$ be a $Y(\mathfrak{g l}_{4})$-singular vector, $\xi_1$, $\xi_2$, $\xi_3$ be nonnegative integers, and
$\boldsymbol{t}=\bigl(t^1_1,\dots,t^1_{\xi_1};t^2_1,\dots,t^2_{\xi_2};t^3_1,\dots,t^3_{\xi_3}\bigr)$. Then
\label{propsplitting}
\begin{equation}\label{splitting}
\mathbb{B}_{\boldsymbol{\xi}}(\boldsymbol{t})v =
\sum_{\boldsymbol{a},\boldsymbol{b}} \bigl(\mathcal{T}\bigl(\boldsymbol{t}^{2}\bigr)\bigr)^{\boldsymbol{a}}_{\boldsymbol{b}}
\bigl(\phi_2\bigl(\boldsymbol{t}^ 2\bigr)\bigl(\mathbb{B}_{\xi_1}^{\ltwo}\bigl(\boldsymbol{t}^{1}\bigr)\bigr)\bigr)^{\boldsymbol{a}}
\bigl(\psi_2\bigl(\boldsymbol{t}^{2}\bigr)\bigl(\mathbb{B}_{\xi_3}^{\ltwo}\bigl(\boldsymbol{t}^{ 3}\bigr)\bigr)\bigr)^{\boldsymbol{b-2}} v,
\end{equation}
where the sum is taken over all sequences $\boldsymbol{a}=(a_1,a_2,\dots,a_{\xi_2})$, $\boldsymbol{b}=(b_1,b_2,\dots,b_{\xi_2})$, such that $a_i\in \{1,2\}$, $b_i\in \{3,4\}$ for all $i=1,\dots,\xi_2$, $\boldsymbol{b-2}=(b_1-2,b_2-2,\dots,b_{\xi_2}-2)$, and
\begin{equation*}
\bigl( \mathcal{T}\bigl(\boldsymbol{t}^2\bigr)\bigr)^{\boldsymbol{a}}_{\boldsymbol{b}}=T\bigl(t^2_1\bigr)^{a_1}_{b_1}T\bigl(t^2_2\bigr)^{a_2}_{b_2}\cdots T\bigl(t^2_{\xi_2}\bigr)^{a_{\xi_2}}_{b_{\xi_2}}.
\end{equation*}
\end{Proposition}

\begin{proof}
Formula ($\ref{splitting}$) follows from the definition of the maps $\psi_2\bigl(\boldsymbol{t}^{2}\bigr)$ and $\phi_2\bigl(\boldsymbol{t}^{2}\bigr)$ and Lemma~$\ref{lemma}$ below.
\end{proof}

\begin{Lemma}\label{lemma} One has
\begin{equation*}
\mathbb{B}_{\boldsymbol{\xi}}(\boldsymbol{t})v=
\sum_{\boldsymbol{a},\boldsymbol{b}} \bigl( \mathcal{T}\bigl(\boldsymbol{t}^{ 2}\bigr) \bigr)^{\boldsymbol{a}}_{\boldsymbol{b}} \Bigl( \overset{[21]}{\mathbb{R}}\bigl(\boldsymbol{t}^{ 2},\boldsymbol{t}^{ 1}\bigr) \overset{[1]}{\mathbb{T}}\bigl(\boldsymbol{t}^{ 1}\bigr) \Bigr)^{\boldsymbol{1}^{\xi_1}, \boldsymbol{2}^{\xi_2}, \boldsymbol{3}^{\xi_3}}_{\boldsymbol{2}^{\xi_1}, \boldsymbol{a}, \boldsymbol{3}^{\xi_3}} \Bigl( \overset{[3]}{\mathbb{T}}\bigl(\boldsymbol{t}^{ 3}\bigr) \overset{[32]}{\mathbb{R}}\bigl(\boldsymbol{t}^{ 3},\boldsymbol{t}^{ 2}\bigr) \Bigr)^{\boldsymbol{1}^{\xi_1}, \boldsymbol{b}, \boldsymbol{3}^{\xi_3}}_{\boldsymbol{1}^{\xi_1}, \boldsymbol{3}^{\xi_2}, \boldsymbol{4}^{\xi_3}}v,
\end{equation*}
where the sum over $\boldsymbol{a}$, $\boldsymbol{b}$ is the same as in formula $\eqref{splitting}$.
\end{Lemma}

\begin{proof}
Using the Yang--Baxter equation~\eqref{YB}, we can write ${\mathbb{B}}(\boldsymbol{t})$ in the following form:
\begin{equation*}
{\mathbb{B}}_{\boldsymbol{\xi}}(\boldsymbol{t})v=\Bigl( \overset{[21]}{\mathbb{R}}\bigl(\boldsymbol{t}^2,\boldsymbol{t}^1\bigr)\overset{[2]}{\mathbb{T}}\bigl(\boldsymbol{t}^2\bigr)\overset{[1]}{\mathbb{T}}\bigl(\boldsymbol{t}^1\bigr) \overset{[3]}{\mathbb{T}}\bigl(\boldsymbol{t}^3\bigr)\overset{[31]}{\mathbb{R}}\bigl(\boldsymbol{t}^3,\boldsymbol{t}^1\bigr)\overset{[32]}{\mathbb{R}} \bigl(\boldsymbol{t}^3,\boldsymbol{t}^2\bigr)\Bigr)^{\boldsymbol{1}^{\xi_1}, \boldsymbol{2}^{\xi_2}, \boldsymbol{3}^{\xi_3}}_{\boldsymbol{2}^{\xi_1}, \boldsymbol{3}^{\xi_2}, \boldsymbol{4}^{\xi_3}}v.
\end{equation*}
Therefore,
\begin{equation}\label{begin}
{\mathbb{B}}_{\boldsymbol{\xi}}(\boldsymbol{t})v=\sum_{\boldsymbol{p},\boldsymbol{q},\boldsymbol{r},\boldsymbol{s}}\Bigl( \overset{[21]}{\mathbb{R}} \overset{[2]}{\mathbb{T}} \overset{[1]}{\mathbb{T}} \Big)^{\boldsymbol{1}^{\xi_1},\boldsymbol{2}^{\xi_2}, \boldsymbol{3}^{\xi_3}}_{\boldsymbol{p},\boldsymbol{q},\boldsymbol{3}^{\xi_3}} \Bigl( \overset{[3]}{\mathbb{T}} \Big)^{\boldsymbol{p},\boldsymbol{q},\boldsymbol{3}^{\xi_3}}_{\boldsymbol{p},\boldsymbol{q},\boldsymbol{r}} \Bigl(\overset{[31]}{\mathbb{R}} \Big)^{\boldsymbol{p},\boldsymbol{q},\boldsymbol{r}}_{\boldsymbol{2}^{\xi_1},\boldsymbol{q},\boldsymbol{s}} \Bigl( \overset{[32]}{\mathbb{R}} \Big)^{\boldsymbol{2}^{\xi_1},\boldsymbol{q},\boldsymbol{s}}_{\boldsymbol{2}^{\xi_1},\boldsymbol{3}^{\xi_2},\boldsymbol{4}^{\xi_3}}v,
\end{equation}
where $\boldsymbol{p}=(p_1,\dots, p_{\xi_1})$, $\boldsymbol{q}=(q_1,\dots, q_{\xi_2})$, $\boldsymbol{r}=(r_1,\dots, r_{\xi_3})$, $\boldsymbol{s}=(s_1,\dots, s_{\xi_3})$. In $\eqref{begin}$, we~omitted the arguments $\boldsymbol{t}^{1}$, $\boldsymbol{t}^{2}$, $\boldsymbol{t}^{3}$ since they can be restored from the context.

We say that \smash{$\boldsymbol r\geq \boldsymbol{3}^{\xi_3}$} if $r_i\geq 3$ for all $i=1,\dots, \xi_3$. Observe that by the definition of a singular vector and the commutation relations
\begin{equation*}
T^{3}_b(w)T^3_d(u)=\frac{w-u-1}{w-u} T^3_d(u)T^3_b(w)+\frac{1}{w-u}T^3_d(w)T^3_b(u),
\end{equation*}
we have \smash{$\overset{[3]}{\mathbb{T}}\bigl(\boldsymbol{t}^3\bigr)_{\boldsymbol{r}}^{\boldsymbol{3}^{\xi_3}}v=0$} unless $\boldsymbol{r}\geq \boldsymbol{3}^{\xi_3}$.

Furthermore, for \smash{$\boldsymbol{r}\geq \boldsymbol {3}^{\xi_3}$}, we have by induction on $\xi_3$ that
\[
\Bigl(\overset{[31]}{\mathbb{R}} \Big)^{\boldsymbol{p},\boldsymbol{q},\boldsymbol{r}}_{\boldsymbol{2}^{\xi_1}, \boldsymbol{q},\boldsymbol{s}}=\delta_{\boldsymbol{p},\boldsymbol{2}^{\xi_1}} \delta_{\boldsymbol{r},\boldsymbol{s}}.
\]
Indeed, for $\xi_3=0$, the statement is true. Assume that
\[
\Bigl(\overset{[31]}{\mathbb{R}} \Big)^{\boldsymbol{p},\boldsymbol{q},\boldsymbol{r}}_{\boldsymbol{2}^{\xi_1},\boldsymbol{q},\boldsymbol{s}}=\delta_{\boldsymbol{p}, \boldsymbol{2}^{\xi_1}} \delta_{\boldsymbol{r},\boldsymbol{s}}
\]
for $\boldsymbol{r}\geq \boldsymbol {3}^{\xi_3}$ if $\xi_3=n-1$, and consider the case $\xi_3=n$.
Let $\boldsymbol{r}=(r_1,\dots,r_n)$, $\boldsymbol{s}=(s_1,\dots,s_n)$, $\tilde{\boldsymbol{r}}=(r_1,\dots,r_{n-1})$, and $\tilde{\boldsymbol{s}}=(s_1,\dots,s_{n-1})$. Then
\begin{equation}\label{30}
\Bigl(\overset{[31]}{\mathbb{R}} \Big)^{\boldsymbol{p},\boldsymbol{q},\boldsymbol{r}}_{\boldsymbol{2}^{\xi_1},\boldsymbol{q},\boldsymbol{s}}=\sum_{\boldsymbol{x}}
\Biggl( \overset{\rightarrow}{\prod\limits_{1\leq i\leq n-1}} \Biggl( \overset{\leftarrow}{\prod\limits_{1\leq j\leq \xi_1}} R^{(\xi^2+i,j)} \Biggr)\Biggr)^{\boldsymbol{p},\boldsymbol{q},\tilde{\boldsymbol{r}},r_n}_{\boldsymbol{x},\boldsymbol{q},\tilde{\boldsymbol{s}},r_n}
\Biggl( \overset{\leftarrow}{\prod\limits_{1\leq k\leq \xi_1}}R^{( \xi^3,k)}\Biggr)^{\boldsymbol{x},\boldsymbol{q},\tilde{\boldsymbol{s}},r_n}_{\boldsymbol{2}^{\xi_1},\boldsymbol{q},\tilde{\boldsymbol{s}},s_n}.
\end{equation}
Observe that the $R$-matrix entry \smash{$R^{jl}_{ik}$} with $i\neq l$ is not zero if and only if $i=j$ and $k=l$, and \smash{$R^{ik}_{ik}=1$}. Because of that and since $r_n\geq 3$, the last factor \smash{$\bigl(\prod R^{(\xi^3,k)} \bigr)^{\boldsymbol{x},\boldsymbol{q},\tilde{\boldsymbol{s}},{r_n}}_{\boldsymbol{2}^{\xi_1},\boldsymbol{q},\tilde{\boldsymbol{s}},s_n}$} in $\eqref{30}$ equals $\delta_{\boldsymbol{x},\boldsymbol{2}^{\xi_1}} \delta_{r_n,s_n}$, and we get
\begin{equation*}
\Bigl(\overset{[31]}{\mathbb{R}} \Big)^{\boldsymbol{p},\boldsymbol{q},\boldsymbol{r}}_{\boldsymbol{2}^{\xi_1},\boldsymbol{q},\boldsymbol{s}}=
\Biggl( \overset{\rightarrow}{\prod\limits_{1\leq i\leq n-1}} \Bigg(\overset{\leftarrow}{\prod\limits_{1\leq j\leq \xi_1}}R^{(\xi^2+i,j)} \Bigg)\Biggr)^{\boldsymbol{p},\boldsymbol{q},\tilde{\boldsymbol{r}},r_n}_{\boldsymbol{2}^{\xi_1},\boldsymbol{q},\tilde{\boldsymbol{s}},r_n}\delta_{r_n,s_n}
= \delta_{\boldsymbol{p},\boldsymbol{2}^{\xi_1}} \delta_{\boldsymbol{\tilde{r}},\boldsymbol{\tilde{s}}} \delta_{r_n,s_n}
= \delta_{\boldsymbol{p},\boldsymbol{2}^{\xi_1}} \delta_{\boldsymbol{{r}},\boldsymbol{{s}}},
\end{equation*}
by the induction assumption.

Since
\[
\Bigl( \overset{[3]}{\mathbb{T}} \Big)^{\boldsymbol{p},\boldsymbol{q},\boldsymbol{3}^{\xi_3}}_{\boldsymbol{p},\boldsymbol{q},\boldsymbol{r}} v=0
\]
 unless $\boldsymbol{r}\geq \boldsymbol{3}^{\xi_3}$ and
\[
\Bigl(\overset{[31]}{\mathbb{R}} \Big)^{\boldsymbol{p},\boldsymbol{q},\boldsymbol{r}}_{\boldsymbol{2}^{\xi_1}, \boldsymbol{q},\boldsymbol{s}}=\delta_{\boldsymbol{p},\boldsymbol{2}^{\xi_1}} \delta_{\boldsymbol{{r}},\boldsymbol{{s}}}
\]
for $\boldsymbol{r}\geq \boldsymbol{3}^{\xi_3}$, formula $\eqref{begin}$ becomes
\begin{equation*}
{\mathbb{B}}_{\boldsymbol{\xi}}(\boldsymbol{t})v=\sum_{\boldsymbol{q},\boldsymbol{r}}
\Bigl( \overset{[21]}{\mathbb{R}} \overset{[2]}{\mathbb{T}} \overset{[1]}{\mathbb{T}} \Big)^{\boldsymbol{1}^{\xi_1},\boldsymbol{2}^{\xi_2}, \boldsymbol{3}^{\xi_3}}_{\boldsymbol{2}^{\xi_1},\boldsymbol{q},\boldsymbol{3}^{\xi_3}}
\Bigl( \overset{[3]}{\mathbb{T}} \Big)^{\boldsymbol{2}^{\xi_1},\boldsymbol{q},\boldsymbol{3}^{\xi_3}}_{\boldsymbol{2}^{\xi_1},\boldsymbol{q},\boldsymbol{r}} \Bigl( \overset{[32]}{\mathbb{R}} \Big)^{\boldsymbol{2}^{\xi_1},\boldsymbol{q},\boldsymbol{r}}_{\boldsymbol{2}^{\xi_1},\boldsymbol{3}^{\xi_2},\boldsymbol{4}^{\xi_3}}v,
\end{equation*}
and can be further transformed as
\begin{equation}\label{33}
{\mathbb{B}}_{\boldsymbol{\xi}}(\boldsymbol{t})v= \sum_{\boldsymbol{a},\boldsymbol{b},\boldsymbol{c},\boldsymbol{r}}
\Bigl(\overset{[2]}{\mathbb{T}} \Big)^{\boldsymbol{c},\boldsymbol{a},\boldsymbol{3}^{\xi_3}}_{\boldsymbol{c},\boldsymbol{b},\boldsymbol{3}^{\xi_3}}
\Bigl(\overset{[21]}{\mathbb{R}} \Big)^{\boldsymbol{1}^{\xi_1},\boldsymbol{2}^{\xi_2},\boldsymbol{3}^{\xi_3}}_{\boldsymbol{c},\boldsymbol{a},\boldsymbol{3}^{\xi_3}}
\Bigl(\overset{[1]}{\mathbb{T}} \Big)^{\boldsymbol{c},\boldsymbol{b},\boldsymbol{3}^{\xi_3}}_{\boldsymbol{2}^{\xi_1},\boldsymbol{b},\boldsymbol{3}^{\xi_3}} \Bigl( \overset{[3]}{\mathbb{T}} \Big)^{\boldsymbol{2}^{\xi_1},\boldsymbol{b},\boldsymbol{3}^{\xi_3}}_{\boldsymbol{2}^{\xi_1},\boldsymbol{b},\boldsymbol{r}} \Bigl( \overset{[32]}{\mathbb{R}} \Big)^{\boldsymbol{2}^{\xi_1},\boldsymbol{b},\boldsymbol{r}}_{\boldsymbol{2}^{\xi_1},\boldsymbol{3}^{\xi_2},\boldsymbol{4}^{\xi_3}}v,
\end{equation}
where the sum is over all sequences $\boldsymbol{a}=(a_1,\dots,a_{\xi_2})$, $\boldsymbol{b}=(b_1,\dots,b_{\xi_2})$, $\boldsymbol{c}=(c_1,\dots, c_{\xi_1})$, $\boldsymbol{r}=(r_1,\dots,r_{\xi_3})$ such that $a_i,b_i,c_i,r_i\in\{1,2,3,4\}$. Since
\[
\Bigl(\overset{[21]}{\mathbb{R}} \Bigr)^{\boldsymbol{1}^{\xi_1},\boldsymbol{2}^{\xi_2},\boldsymbol{3}^{\xi_3}}_{\boldsymbol{c},\boldsymbol{a},\boldsymbol{3}^{\xi_3}}=0
\]
 if $a_i\geq3$ for some~$i$, and
 \[
 \Bigl( \overset{[32]}{\mathbb{R}} \Bigr)^{\boldsymbol{2}^{\xi_1},\boldsymbol{b},\boldsymbol{r}}_{\boldsymbol{2}^{\xi_1},\boldsymbol{3}^{\xi_2},\boldsymbol{4}^{\xi_3}}=0
 \]
 if $b_i\leq 2$ for some $i$, terms in the sum on the right-hand side of $\eqref{33}$ equal zero unless $a_i\in\{1,2\}$ and $b_i\in \{3,4\}$ for all $i$. Taking the sum over $\boldsymbol{c}$ and $\boldsymbol{r}$ in formula~\eqref{33}, we get the statement of Lemma~\ref{lemma}.
\end{proof}

\begin{Example} Here we illustrate the proof of the relation
\[
\Bigl(\overset{[31]}{\mathbb{R}} \Big)^{\boldsymbol{p},\boldsymbol{q},\boldsymbol{r}}_{\boldsymbol{2}^{\xi_1}, \boldsymbol{q},\boldsymbol{s}}=\delta_{\boldsymbol{p},\boldsymbol{2}^{\xi_1}} \delta_{\boldsymbol{r},\boldsymbol{s}}
\]
 if $\boldsymbol{r}\geq \boldsymbol {3}^{\xi_3}$ for $\xi_1=\xi_3=2$. In this case, $ \boldsymbol{p}=(p_1,p_2)$, $\boldsymbol{r}=(r_1,r_2)$, $\boldsymbol{s}=(s_1,s_2)
$, and
\begin{equation*}
\Bigl(\overset{[31]}{\mathbb{R}} \Big)^{\boldsymbol{p},\boldsymbol{q},\boldsymbol{r}}_{\boldsymbol{2}^{\xi_1},\boldsymbol{q},\boldsymbol{s}}= \sum_{a,b,c,d}{R}^{p_2r_1}_{ab}\bigl(t^3_1-t^1_2\bigr) {R}^{p_1 b}_{cs_1}\bigl(t^3_1-t^1_1\bigr) {R}^{ar_2}_{2d}\bigl(t^3_2-t^1_2\bigr) {R}^{cd}_{2s_2}\bigl(t^3_2-t^1_1\bigr).
\end{equation*}
For $r_1\geq 3$, $r_2\geq 3$, we have ${R}^{ar_2}_{2d}\bigl(t^3_2-t^1_2\bigr)=\delta_{a,2} \delta_{r_2,d}$, thus
\begin{equation*}
\Bigl(\overset{[31]}{\mathbb{R}} \Big)^{\boldsymbol{p},\boldsymbol{q},\boldsymbol{r}}_{\boldsymbol{2}^{\xi_1},\boldsymbol{q},\boldsymbol{s}}= \sum_{b,c}{R}^{p_2r_1}_{2b}\bigl(t^3_1-t^1_2\bigr) {R}^{p_1 b}_{cs_1}\bigl(t^3_1-t^1_1\bigr) {R}^{cr_2}_{2s_2}\bigl(t^3_2-t^1_1\bigr).
\end{equation*}
Then ${R}^{cr_2}_{2s_2}\bigl(t^3_2-t^1_1\bigr)=\delta_{c,2} \delta_{r_2,s_2}$ and ${R}^{p_2r_1}_{2b}\bigl(t^3_1-t^1_2\bigr)=\delta_{p_2,2} \delta_{r_1,b}$, so that
\begin{equation*}
\Bigl(\overset{[31]}{\mathbb{R}} \Big)^{\boldsymbol{p},\boldsymbol{q},\boldsymbol{r}}_{\boldsymbol{2}^{\xi_1},\boldsymbol{q},\boldsymbol{s}}= {R}^{p_1 r_1}_{2s_1}\bigl(t^3_1-t^1_1\bigr) \delta_{p_2,2} \delta_{r_2,s_2} =\delta_{p_2,2} \delta_{p_1,2} \delta_{r_1,s_1} \delta_{r_2,s_2}=\delta_{\boldsymbol{p},\boldsymbol{2}^{\xi_1}} \delta_{\boldsymbol{r},\boldsymbol{s}}.
\end{equation*}
\end{Example}

\section[Main theorem for the gl\_4 case]{Main theorem for the $\boldsymbol{\mathfrak{gl}_4}$ case}

The main result of this paper is Theorem~\ref{main2} formulated at the end of this section. We will approach it in several steps.

For a nonnegative integer $m$, set
\begin{equation}\label{Qfun}
Q_{m} (t_{1}, \dots, t_{m} )=\prod_{1 \leqslant i<j \leqslant m} \frac{t_{i}-t_{j}-1}{t_{i}-t_{j}}.
\end{equation}
For an expression $f (t_{1},\dots,t_m )$, define
\begin{equation*}
\Sym_{\boldsymbol{t}} f (t_{1}, \dots, t_m )=\sum_{\sigma\in S_m} f \bigl(t_{\sigma(1)}, \dots, t_{\sigma(m)} \bigr),
\end{equation*}
and
\begin{equation}\label{barsym}
\Symb_{\boldsymbol{t}} f(t_{1}, \dots, t_m)=
\Sym_{\boldsymbol{t}} (f(t_{1}, \dots, t_m) Q_m (t_{1}, \dots, t_m ) ).
\end{equation}

To simplify notation, we will write \smash{$T^{\ltwo}_{ij}$} instead of \smash{$\bigl(T^{\ltwo}\bigr)^i_j$}.

\begin{Proposition}\label{twotensor}
Let $\xi$ be a nonnegative integer and $\boldsymbol{t}=$ $(t_{1}, \dots, t_{\xi} )$. Then
\begin{gather*}
\Delta \bigl( \mathbb{B}^{\ltwo}_{\xi}(\boldsymbol{t}) \bigr)=
\sum_{\eta=0}^{\xi} \frac{1}{(\xi-\eta) ! \eta!}\\
\qquad\times
\Symb_{\boldsymbol{t}}\Biggl( \bigl(\mathbb{B}^{\ltwo}_{\eta}(t_1,\dots,t_\eta) \otimes \mathbb{B}^{\ltwo}_{\xi-\eta}(t_{\eta+1},\dots,t_{\xi}) \bigr)
 \Biggl(\prod_{i=\eta+1}^{\xi } T^{\ltwo}_{22}(t_{i})\otimes \prod_{j=1}^{\eta} T^{\ltwo}_{11}(t_j)\Biggr) \Biggr).
\end{gather*}
\end{Proposition}
This proposition goes back to \cite{Kor,TV}. For convenience, we give its proof in Appendix~\ref{appendixA}.

Given a subset $I $ of $ \{1,2,\dots, k\}$ denote by $I^*$ the complement of $I$ in $\{1,2,\dots,k\}$. Define a~vector \smash{$\mathbf{w}^{I}\in \bigl(\mathbb{C}^2\bigr)^{\otimes k}$} by the rule{\samepage
\begin{equation*} \mathbf{w}^{I}=\mathbf{w}_{a_1}\otimes \mathbf{w}_{a_2}\otimes\dots\otimes \mathbf{w}_{a_k},\end{equation*}
where $a_i=2$ if $i\in I$, and $a_i=1$ if $i\not\in I$.}

Fix a $Y(\mathfrak{g l}_{2})$-module $V$ and a weight singular vector $v \in V$ with respect to the $Y(\mathfrak{g l}_{2})$-action,
\begin{equation*}
T^{\ltwo}_{21}(u)v=0, \qquad T^{\ltwo}_{11}(u)v=\bigl\langle T^{\ltwo}_{11}(u)v \bigr\rangle v,\qquad T^{\ltwo}_{22}(u)v=\bigl\langle T^{\ltwo}_{22}(u)v \bigr\rangle v.
\end{equation*}
Here \smash{$\bigl\langle T^{\ltwo}_{11}(u)v \bigr\rangle$} and \smash{$\bigl\langle T^{\ltwo}_{22}(u)v \bigr\rangle$} are the corresponding eigenvalues.
Given complex numbers $z_1,\dots,z_k$, consider the $Y(\mathfrak{g l}_{2})$-module $L(z_1) \otimes \cdots \otimes L(z_k) \otimes V $. Observe that \smash{$\mathbf{w}_{1}^{\otimes k}\otimes v$} is a weight singular vector with respect to the action of $Y(\mathfrak{g l}_{2})$ in this module.

\begin{Proposition}\label{w1} For the $Y(\mathfrak{g l}_{2})$-module $L(z_1) \otimes \cdots \otimes L(z_k) \otimes V $, we have
\begin{gather}
\mathbb{B}^{\ltwo}_\xi(\boldsymbol{t})\bigl(\mathbf{w}_{1}^{\otimes k}\otimes v\bigr)\nonumber\\
\qquad{}=\sum_{I}\frac{1}{(\xi-|I|)!} {\Symb_{\boldsymbol{t}}}\Biggl[F_{I}(\boldsymbol{t},\boldsymbol{z}) \prod_{a=1}^{|I|} \bigl\langle T^{\ltwo}_{11}(t_{a} ) v \bigr\rangle \bigl( \mathbf{w}^{I}\otimes \mathbb{B}^{\ltwo}_{\xi-|I|}\bigl(t_{|I|+1},\dots,t_{\xi}\bigr) v\bigr) \Biggr],\label{formulaw1}
\end{gather}
where the sum is over all subsets $I\subset \{1,\dots,k\}$ such that $|I|\leq\xi$, and for a given $I=\{i_1<i_2<\dots<i_{|I|}\}$,
\begin{equation}\label{F}
F_{I}(\boldsymbol{t},\boldsymbol{z})= \prod_{a=1}^{|I|} \Biggl(\frac{1}{t_a-z_{i_a}}\prod_{ m=i_a+1}^k \frac{t_a-z_m+1}{t_a-z_m}\Biggr).
\end{equation}
\end{Proposition}
\begin{proof}
Observe that for each $Y(\mathfrak{g l}_{2})$-module $L(z_i)$, $i=1,\dots, k$, the corresponding vector ${\mathbf{w}_{1}\in L(z_i)}$ is a weight singular vector,
\begin{equation*}
T^{\ltwo}_{11}(u) \mathbf{w}_{1} =\bigl(1+(u-z_i)^{-1}\bigr)\mathbf{w}_{1},\qquad T^{\ltwo}_{22}(u) \mathbf{w}_{1} =\mathbf{w}_{1},\qquad T^{\ltwo}_{21}(u) \mathbf{w}_{1}=0.
\end{equation*}
Moreover, \smash{$\mathbb{B}_{1}^{\ltwo}(u) \mathbf{w}_{1}=T^{\ltwo}_{12}(u)\mathbf{w}_{1}=(u-z_i)^{-1} \mathbf{w}_{2}$}
and \smash{$ \mathbb{B}_{\zeta}^{\ltwo}(u_1,\dots,u_\zeta) \mathbf{w}_{1}=0$} for $\zeta\geq 2$. Then formula $\eqref{formulaw1}$ follows from Proposition $\ref{twotensor}$ and identity~\eqref{n!} by induction on $k$.
\end{proof}

Given complex numbers $z_1,\dots,z_k$, consider the $Y(\mathfrak{g l}_{2})$-module $V\otimes \bar{L}(z_k) \otimes \cdots \otimes \bar{L}(z_1) $. Observe that \smash{$v\otimes \mathbf{w}_{2}^{\otimes k} $} is a weight singular vector with respect to the action of $Y(\mathfrak{g l}_{2})$ in this module.

\begin{Proposition}\label{w2}For the $Y(\mathfrak{g l}_{2})$-module $V\otimes \bar{L}(z_k) \otimes \cdots \otimes \bar{L}(z_1)$, we have
\begin{gather}
\mathbb{B}^{\ltwo}_\xi(\boldsymbol{t})\bigl(v\otimes \mathbf{w}_{2}^{\otimes k} \bigr)\label{formulaw2}\\
\qquad{}=\sum_{I}\frac{1}{(\xi-|I|)!} {\Symb_{\boldsymbol{t}}}\Biggl[\widetilde{F}_{I}(\boldsymbol{t},\boldsymbol{z}) \prod_{i=1}^{|I|} \bigl\langle T^{\ltwo}_{22}\bigl(t_{\xi-|I|+i} \bigr) v \bigr\rangle \bigl(\mathbb{B}^{\ltwo}_{\xi-|I|}\bigl(t_1,\dots,t_{\xi-|I|}\bigr) v \otimes \mathbf{w}^{I^*}\bigr) \Biggr],\nonumber
\end{gather}
where the sum is over all subsets $I\subset \{1,\dots,k\}$ such that $|I|\leq\xi$, and for a given $I=\{i_1<i_2<\dots<i_{|I|}\}$,
\begin{equation}\label{Ft}
\widetilde{F}_{I}(\boldsymbol{t},\boldsymbol{z})= \prod_{a=1}^{| I|} \Biggl( \frac{1}{z_{i_a}-t_{\xi-a+1}}\prod_{ m=i_a+1}^k \frac{z_m-t_{\xi-a+1}+1}{z_m-t_{\xi-a+1}}\Biggr).
\end{equation}
\end{Proposition}
\begin{proof}

Observe that for each $Y(\mathfrak{g l}_{2})$-module $\bar{L}(z_i)$, $i=1,\dots, k$, the corresponding vector ${\mathbf{w}_{2}\in \bar{L}(z_i)}$ is a weight singular vector,
\begin{equation*}
T^{\ltwo}_{11}(u) \mathbf{w}_{2} =\mathbf{w}_{2},\qquad \ T^{\ltwo}_{22}(u) \mathbf{w}_{2} =\bigl(1+(z_i-u)^{-1}\bigr)\mathbf{w}_{2},\qquad T^{\ltwo}_{21}(u) \mathbf{w}_{2}=0.
\end{equation*}
Moreover, \smash{$\mathbb{B}_{1}^{\ltwo}(u) \mathbf{w}_{2}=T^{\ltwo}_{12}(u)\mathbf{w}_{2}=(z_i-u)^{-1} \mathbf{w}_{1}$},
and \smash{$ \mathbb{B}_{\zeta}^{\ltwo}(u_1,\dots,u_\zeta) \mathbf{w}_{2}=0$} for $\zeta\geq 2$.
Then formula $\eqref{formulaw2}$ follows from Proposition $\ref{twotensor}$ and identity~\eqref{n!} by induction on $k$.
\end{proof}

For $\boldsymbol{t}=(t_1,\dots,t_\xi)$, $\boldsymbol{z}=(z_1,\dots,z_k)$, $y\in \mathbb{C}$, and a subset $I=\{i_1<i_2<\dots<i_{|I|}\}\subset \{1,\dots,k\}$, define the functions
\begin{equation}\label{V}
V_{I}(\boldsymbol{t}, \boldsymbol{z},y)=\frac{1}{(\xi-|I|) !} \Symb_{\boldsymbol{t}}\Biggl( F_I(\boldsymbol{t},\boldsymbol{z}) \prod_{a=1}^{|I|}(t_a-y)\Biggr)
\end{equation}
and
\begin{equation}\label{Vt}
\widetilde{V}_{I}(\boldsymbol{t}, \boldsymbol{z},y)=\frac{1}{(\xi-|I|) !} \Symb_{\boldsymbol{t} }\Biggl( \widetilde{F}_I(\boldsymbol{t},\boldsymbol{z}) \prod_{a=1}^{|I|}(t_{\xi-a+1}-y)\Biggr).
\end{equation}

Consider the collection $\mathcal{S}_{p,q,r,k}$ of pairs of subsets of $\{1,\dots,k\}$ with given cardinalities of the subsets and their intersection. Namely,
\begin{equation*}
\mathcal{S}_{p,q,r,k}=\{(I,J) \mid I, J \subset\{1,\dots,k\}, \, |I|=p,\, |J|=q,\, |I\cap J| =r\}.
\end{equation*}
For $I\subset\{1,\dots,k\}$, set $\check{I}=\{k-i+1, i\in I\}$.
\begin{Theorem}\label{main}
Let $V$ be a $\mathfrak{gl}_4$-module and $v\in V$ a $\mathfrak{gl}_4$-singular vector of weight $\bigl(\Lambda^1,\Lambda^2,\Lambda^3,\Lambda^4\bigr)$. Let $\xi_1$, $\xi_2$, $\xi_3$ be nonnegative integers, $\boldsymbol{t}^1=\bigl(t^1_1,\dots,t^1_{\xi_1}\bigr)$, $\boldsymbol{t}^2=\bigl(t^2_1,\dots,t^2_{\xi_2}\bigr)$, $\boldsymbol{t}^3=\bigl(t^3_1,\dots,t^3_{\xi_3}\bigr)$, and $\boldsymbol{t}=\bigl(\boldsymbol{t}^1,\boldsymbol{t}^2,\boldsymbol{t}^3\bigr)$. For every triple $(p,q,r)$, $p=0,\dots,\min(\xi_2,\xi_3)$, $q=0,\dots,\min(\xi_2,\xi_1)$, $r=\max (0, p+q-\xi_2),\dots,\min(p,q)$, fix a pair $(I_{p,q,r},J_{p,q,r})\in \mathcal{S}_{p,q,r,\xi_2}$. Then,

\begin{enumerate}\itemsep=0pt
\item[$(a)$] In the evaluation $Y(\mathfrak{gl}_4)$-module $V(x)$, one has
\begin{gather}
 {\mathbb{B}}_{\boldsymbol{\xi}}(\boldsymbol{t}) v=\prod_{a=1}^{3}\prod_{i=1}^{\xi_a}\frac{1}{t^a_i-x}\nonumber\\
\times\sum\limits_{p=0}^{\min(\xi_2,\xi_3)}
\sum\limits_{q=0}^{\min(\xi_2,\xi_1)} \!\!\!\sum\limits_{r=\max (0, p+q-\xi_2)}^{\min(p,q)} \Biggl(\Symb_{\boldsymbol{t}^2 }\bigl( V_{\check{I}_{p,q,r}}\bigl(\boldsymbol{t}^3, \boldsymbol{t}^2,x-\Lambda^3\bigr) \widetilde{V}_{J_{p,q,r}}\bigl( \boldsymbol{t}^1, \check{\boldsymbol{t}}^2,x-\Lambda^2\bigr)\bigr) \nonumber\\
\hphantom{\sum\limits_{p=0}^{\min(\xi_2,\xi_3)}\sum\limits_{q=0}^{\min(\xi_2,\xi_1)} \sum\limits_{r=\max (0, p+q-\xi_2)}^{\min(p,q)} \Biggl(}{} \
\times \frac{e_{32}^{\xi_2-p-q+r} e_{31}^{q-r} e_{42}^{p-r} e_{41}^{r} e_{21}^{\xi_1-q} e_{43}^{\xi_3-p} v}{(p-r)!(q-r)! r!(\xi_2-p-q+r)!}\Biggr),\label{mainTh}
\end{gather}
where $\check{\boldsymbol{t}}^2=\bigl(t^2_{\xi_2},\dots,t^2_1\bigr)$.
\item[$(b)$] The function \smash{$\Symb_{\boldsymbol{t}^2 }\bigl(V_{\check{I}_{p,q,r}}\bigl(\boldsymbol{t}^{ 3}, \boldsymbol{t}^{ 2},x-\Lambda^3\bigr) \widetilde{V}_{J_{p,q,r}}\bigl(\boldsymbol{t}^{ 1}, \check{\boldsymbol{t}}^{ 2},x-\Lambda^2\bigr) \bigr)$} in $\eqref{mainTh}$ does not depend on the choice of the pair $(I_{p,q,r},J_{p,q,r})$.
\end{enumerate} \end{Theorem}
\begin{proof}
Item (a) follows from Propositions $\ref{prethprop}$ and $\ref{symprop}$ given below. Item (b) is an immediate corollary of Proposition $\ref{symprop}$.

Propositions $\ref{prethprop}$ and $\ref{symprop}$ are proved in Sections $\ref{section5}$ and $\ref{section6}$, respectively.
\end{proof}

\begin{Proposition}\label{prethprop}
In the notation of Theorem $\ref{main}$, we have
\begin{gather}
{\mathbb{B}}_{\boldsymbol{\xi}}(\boldsymbol{t}) v=\prod_{a=1}^{3}\prod_{i=1}^{\xi_a}\frac{1}{t^a_i-x}\nonumber\\
\qquad{}\times\sum\limits_{p=0}^{\min(\xi_2,\xi_3)} \sum\limits_{q=0}^{\min(\xi_2,\xi_1)}\sum\limits_{r=\max (0, p+q-\xi_2)}^{\min(p,q)} \Biggl(\sum\limits_{(I,J)\in \mathcal{S}_{p,q,r,\xi_2}} \widetilde{V}_{J}\bigl(\boldsymbol{t}^{ 1}, \boldsymbol{t}^{ 2},x-\Lambda^2\bigr)
V_{I}\bigl(\boldsymbol{t}^{ 3}, \boldsymbol{t}^{ 2},x-\Lambda^3\bigr)\nonumber\\
\qquad{}\hphantom{\times\sum\limits_{p=0}^{\min(\xi_2,\xi_3)} \sum\limits_{q=0}^{\min(\xi_2,\xi_1)}\sum\limits_{r=\max (0, p+q-\xi_2)}^{\min(p,q)} \Biggl(}{}
\times e_{32}^{\xi_2-p-q+r} e_{31}^{q-r} e_{42}^{p-r} e_{41}^{r} e_{21}^{\xi_1-q} e_{43}^{\xi_3-p} v\Biggr).\label{preth}
\end{gather}
\end{Proposition}

\begin{Proposition}\label{symprop}
In the notation of Theorem $\ref{main}$, we have
\begin{gather}
\sum\limits_{(I,J)\in \mathcal{S}_{p,q,r,\xi_2}} \widetilde{V}_{J}\bigl(\boldsymbol{t}^{ 1}, \boldsymbol{t}^{ 2},x-\Lambda^2\bigr) V_{I}\bigl(\boldsymbol{t}^{ 3}, \boldsymbol{t}^{ 2},x-\Lambda^3\bigr)
\nonumber\\
\qquad{} =\frac{ \Symb_{\boldsymbol{t}^2 } \bigl(V_{\check{I}_0}\bigl(\boldsymbol{t}^3, \boldsymbol{t}^2,x-\Lambda^3\bigr) \widetilde{V}_{J_0}\bigl( \boldsymbol{t}^1, \check{\boldsymbol{t}}^2,x-\Lambda^2\bigr)\bigr) }{(p-r)!(q-r)! r!(\xi_2-p-q+r)!},\label{sympropf}
\end{gather}
where $(I_0,J_0)$ is any pair from $\mathcal S_{p,q,r,\xi_2}$.
\end{Proposition}

Below we reformulate Theorem~\ref{main} in a more closed form.

\begin{Theorem}\label{main2}
Let $V$ be a $\mathfrak{gl}_4$-module and $v\in V$ a $\mathfrak{gl}_4$-singular vector of weight $\bigl(\Lambda^1,\Lambda^2,\Lambda^3,\Lambda^4\bigr)$. Let $\xi_1$, $\xi_2$, $\xi_3$ be nonnegative integers, $\boldsymbol{t}^1=\bigl(t^1_1,\dots,t^1_{\xi_1}\bigr)$, $\boldsymbol{t}^2=\bigl(t^2_1,\dots,t^2_{\xi_2}\bigr)$, $\boldsymbol{t}^3=\bigl(t^3_1,\dots,t^3_{\xi_3}\bigr)$, and $\boldsymbol{t}=\bigl(\boldsymbol{t}^1,\boldsymbol{t}^2,\boldsymbol{t}^3\bigr)$. For every triple $(p,q,r)$, $p=0,\dots,\min(\xi_2,\xi_3)$, $q=0,\dots,\min(\xi_2,\xi_1)$, $r=\max (0, p+q-\xi_2),\dots,\min(p,q)$, fix two sequences $\boldsymbol{i}=\{i_1<\dots<i_p\}$ and $\boldsymbol{j}=\{j_1<\dots<j_q\}$, such that $|\{i_1,\dots,i_p\}\cap \{j_1,\dots,j_q\}|=r$. Then,
\begin{enumerate}\itemsep=0pt
\item[$(a)$] In the evaluation $Y(\mathfrak{gl}_4)$-module $V(x)$, one has
\begin{gather}
{\mathbb{B}}_{\boldsymbol{\xi}}(\boldsymbol{t}) v =\prod_{a=1}^{3}\prod_{i=1}^{\xi_a}\frac{1}{t^a_i-x} \nonumber\\
 \hphantom{{\mathbb{B}}_{\boldsymbol{\xi}}(\boldsymbol{t}) v =}{}
 \times\sum\limits_{p=0}^{\min(\xi_2,\xi_3)} \sum\limits_{q=0}^{\min(\xi_2,\xi_1)}\sum\limits_{r=\max (0, p+q-\xi_2)}^{\min(p,q)}
 \Symb_{\boldsymbol{t}^1} \Symb_{\boldsymbol{t}^2} \Symb_{\boldsymbol{t}^3} G_{\boldsymbol{i},\boldsymbol{j}}(\boldsymbol{t}) \nonumber\\
 \hphantom{{\mathbb{B}}_{\boldsymbol{\xi}}(\boldsymbol{t}) v =}{}
 \qquad\times\frac{e_{32}^{\xi_2-p-q+r} e_{31}^{q-r} e_{42}^{p-r} e_{41}^{r} e_{21}^{\xi_1-q} e_{43}^{\xi_3-p}v}{(\xi_2-p-q+r)!(q-r)!(p-r)! r!(\xi_1-q)!(\xi_3-p)!}, \label{mainTh2}
\end{gather}
where
\begin{align}
G_{\boldsymbol{i},\boldsymbol{j}}(\boldsymbol{t})={}& \prod_{a=1}^{p} \Biggl(\frac{t^3_a-x+\Lambda^3}{t_a^3-t^2_{i_a}}\prod_{ m=i_a+1}^{\xi_2} \frac{t_a^3-t^2_m+1}{t_a^3-t^2_m} \Biggr) \nonumber\\
&{\times}\prod_{s=1}^{q} \Biggl( \frac{t_{\xi_1-q+s}^1-x+\Lambda^2}{t^2_{j_{s}}-t^1_{\xi_1-q+s}}\prod_{ l=1}^{j_s-1} \frac{t^2_l-t_{\xi_1-q+s}^1+1}{t^2_l-t^1_{\xi_1-q+s}} \Biggr).\label{Gfu}
\end{align}
\item[$(b)$] The function $\Symb_{\boldsymbol{t}^1} \hspace{-0.3pt} \Symb_{\boldsymbol{t}^2}\hspace{-0.3pt} \Symb_{\boldsymbol{t}^3}\hspace{-0.3pt} G_{\boldsymbol{i},\boldsymbol{j}}(\boldsymbol{t})$ does not depend on the choice of the sequences~$\boldsymbol{i}$,~$\boldsymbol{j}$.
\end{enumerate}
\end{Theorem}

\begin{proof}
Given the pair $(I_{p,q,r},J_{p,q,r})$ from the formulation of Theorem~\ref{main}, define the sequences $\boldsymbol{i}=\{i_1<\dots<i_p\}$ and $\boldsymbol{j}=\{j_1<\dots<j_q\}$ by the rule
\begin{gather*}
I_{p,q,r}=\{\xi_2-i_1+1,\xi_2-i_2+1,\dots,\xi_2-i_p+1\},\\ J_{p,q,r}=\{\xi_2-j_1+1,\xi_2-j_2+1,\dots,\xi_2-j_q+1\}.
\end{gather*}
Notice that
$\{i_1,\dots,i_p\}=\check{I}_{p,q,r}$, $\{j_1,\dots,j_q\}=\check{J}_{p,q,r}$,
and
\begin{equation*}|\{i_1,\dots,i_p\}\cap \{j_1,\dots,j_q\}|=|I_{p,q,r}\cap J_{p,q,r}|=r.\end{equation*}
Then combining formulae~\eqref{F} and~\eqref{Ft}--\eqref{Vt}, we obtain that
\begin{equation}\label{VIg}
V_{\check{I}_{p,q,r}}\bigl(\boldsymbol{t}^3,\boldsymbol{t}^2,x-\Lambda^3\bigr)=\frac{1}{(\xi_3-p)!}\Symb_{\boldsymbol{t}^3}
\prod_{a=1}^{p} \Biggl(\frac{t^3_a-x+\Lambda^3}{t_a^3-t^2_{i_a}}\prod_{ m=i_a+1}^{\xi_2} \frac{t_a^3-t^2_m+1}{t_a^3-t^2_m}\Biggr)
\end{equation}
and
\begin{gather*}
\widetilde{V}_{J_{p,q,r}}\bigl(\boldsymbol{t}^1,\check{\boldsymbol{t}}^2,x-\Lambda^2\bigr)=\frac{1}{(\xi_1-q)!}\Symb_{\boldsymbol{t}^1}
\prod_{b=1}^{q} \Biggl( \frac{t_{\xi_1-b+1}^1-x+\Lambda^2}{t^2_{j_{q-b+1}}-t^1_{\xi_1-b+1}}\prod_{ l=1}^{j_{q-b+1}-1} \frac{t^2_l-t_{\xi_1-b+1}^1+1}{t^2_l-t^1_{\xi_1-b+1}}\Biggr).
\end{gather*}
After substituting $b=q+1-s$, the last formula becomes
\begin{gather}
\widetilde{V}_{J_{p,q,r}}\bigl(\boldsymbol{t}^1,\check{\boldsymbol{t}}^2,x-\Lambda^2\bigr)\nonumber\\
\qquad{}=\frac{1}{(\xi_1-q)!}\Symb_{\boldsymbol{t}^1}
\prod_{s=1}^{q} \Biggl( \frac{t_{\xi_1-q+s}^1-x+\Lambda^2}{t^2_{j_{s}}-t^1_{\xi_1-q+s}}\prod_{ l=1}^{j_s-1} \frac{t^2_l-t_{\xi_1-q+s}^1+1}{t^2_l-t^1_{\xi_1-q+s}}\Biggr).\label{VJg}
\end{gather}
Plugging~\eqref{VIg} and~\eqref{VJg} into formula~\eqref{mainTh}, we obtain formulae~\eqref{mainTh2} and~\eqref{Gfu}.

Item (b) of Theorem~\ref{main2} is a reformulation of item (b) of Theorem~\ref{main}.
\end{proof}

\begin{Example} Below we give two examples of natural choices of the sequences $\boldsymbol{i}$, $\boldsymbol{j}$ in Theorem~\ref{main2} and write down the corresponding expressions for the function $G_{\boldsymbol{i},\boldsymbol{j}}( \boldsymbol{t})$, see formula~\eqref{Gfu}.
\begin{enumerate}\itemsep=0pt
\item[$(a)$]
$\boldsymbol{i}=\boldsymbol{i}_1=\{1<\dots<p\}$, $\boldsymbol{j}=\boldsymbol{j}_1=\{p+1-r<\dots<p+q-r\}$. Then
\begin{align*}
G_{\boldsymbol{i}_1,\boldsymbol{j}_1}( \boldsymbol{t})={}& \prod_{a=1}^{p}\Biggl( \frac{t^3_a-x+\Lambda^3}{t_a^3-t^2_a}\prod_{ m=a+1}^{\xi_2} \frac{t_a^3-t^2_m+1}{t_a^3-t^2_m}\Biggr)\\
&{\times}\prod_{c=1}^{q} \Biggl( \frac{t_{\xi_1-q+c}^1-x+\Lambda^2}{t^2_{p-r+c}-t^1_{\xi_1-q+c}}\prod_{ l=1}^{p-r+c-1} \frac{t^2_l-t_{\xi_1-q+c}^1+1}{t^2_l-t^1_{\xi_1-q+c}}\Biggr).
\end{align*}

\item[$(b)$] $\boldsymbol{i}=\boldsymbol{i}_2=\{q+1-r< \dots<q+p-r\}$, $\boldsymbol{j}=\boldsymbol{j}_2=\{1<\dots<q\}$. Then
\begin{align*}
G_{\boldsymbol{i}_2,\boldsymbol{j}_2}( \boldsymbol{t})={}& \prod_{a=1}^{p}\Biggl( \frac{t^3_a-x+\Lambda^3}{t_a^3-t^2_{q-r+a}}\prod_{ m=q-r+a+1}^{\xi_2} \frac{t_a^3-t^2_m+1}{t_a^3-t^2_m}\Biggr)\\
&{\times}\prod_{b=1}^{q} \Biggl( \frac{t_{\xi_1-b+1}^1-x+\Lambda^2}{t^2_{b}-t^1_{\xi_1-b+1}} \prod_{ l=1}^{b-1} \frac{t^2_l-t_{\xi_1-b+1}^1+1}{t^2_l-t^1_{\xi_1-b+1}}\Biggr).
\end{align*}
\end{enumerate}
Notice that the equality
\begin{equation*}
\Symb_{\boldsymbol{t}^1} \Symb_{\boldsymbol{t}^2} \Symb_{\boldsymbol{t}^3} G_{\boldsymbol{i}_1,\boldsymbol{j}_1}(\boldsymbol{t}) = \Symb_{\boldsymbol{t}^1} \Symb_{\boldsymbol{t}^2} \Symb_{\boldsymbol{t}^3} G_{\boldsymbol{i}_2,\boldsymbol{j}_2}(\boldsymbol{t}), \end{equation*}
stated in item (b) of Theorem~\ref{main2}, is not obvious.
\end{Example}

\section{Proof of Proposition~\ref{prethprop}}\label{section5}
Let $V$ be a $\mathfrak{gl}_4$-module and $v\in V$ a $\mathfrak{gl}_4$-singular vector of weight $\bigl(\Lambda^1,\Lambda^2,\Lambda^3,\Lambda^4\bigr)$. Let $\xi_1$, $\xi_2$, $\xi_3$ be nonnegative integers, $\boldsymbol{t}^1=\bigl(t^1_1,\dots,t^1_{\xi_1}\bigr)$, $\boldsymbol{t}^2=\bigl(t^2_1,\dots,t^2_{\xi_2}\bigr)$, $\boldsymbol{t}^3=\bigl(t^3_1,\dots,t^3_{\xi_3}\bigr)$ and $\boldsymbol{t}=\bigl(\boldsymbol{t}^1,\boldsymbol{t}^2,\boldsymbol{t}^3\bigr)$.
Recall that in the evaluation $Y(\mathfrak{gl}_4)$-module $V(x)$, we have $T^a_b(u)=\delta_{ab}+ e_{ba}(u-x)^{-1}$, thus
\begin{equation*}T^a_a(u)v=\frac{u-x+\Lambda^a}{u-x} v.\end{equation*}
By Proposition $\ref{propsplitting}$,
\begin{equation}\label{step1}
{\mathbb{B}}_{\boldsymbol{\xi}}(\boldsymbol{t})v =\
\sum_{\boldsymbol{a},\boldsymbol{b}} \Bigl( \overset{[2]}{\mathbb{T}} \bigl(\boldsymbol{t}^{2}\bigr)\Bigr)^{\boldsymbol{a}}_{\boldsymbol{b}} \bigl(\phi_2\bigl(\boldsymbol{t}^ 2\bigr)\bigl(\mathbb{B}_{\xi_1}^{\ltwo}\bigl(\boldsymbol{t}^{1}\bigr)\bigr)\bigr)^{\boldsymbol{a}} \bigl(\psi_2\bigl(\boldsymbol{t}^{2}\bigr)\bigl(\mathbb{B}_{\xi_3}^{\ltwo}\bigl(\boldsymbol{t}^{ 3}\bigr)\bigr)\bigr)^{\boldsymbol{b-2}} v,
\end{equation}
where the sum is taken over all sequences $\boldsymbol{a}=(a_1,a_2,\dots,a_{\xi_2})$, $\boldsymbol{b}=(b_1,b_2,\dots,b_{\xi_2})$, such that $a_i\in \{1,2\}$, and $b_i\in \{3,4\}$ for all $i=1,\dots,\xi_2$.

Let ${ }^{\psi} V(x)$ be the $Y(\mathfrak{g l}_{2})$-module obtained by pulling back the module $V(x)$ through the embedding $\psi_2$.
To compute \smash{$\bigl(\psi_2\bigl(\boldsymbol{t}^{2}\bigr)\bigl(\mathbb{B}_{\xi_3}^{\ltwo}\bigl(\boldsymbol{t}^{ 3}\bigr)\bigr)\bigr)^{\boldsymbol{b}} v$}, we take the weight function \smash{$\mathbb{B}^{\ltwo}_{\xi_3}\bigl(\boldsymbol{t}^3\bigr)\bigl(\mathbf{w}_{1}^{\otimes \xi_2}\otimes v\bigr)$} in~the $Y(\mathfrak{g l}_{2})$-module $L\bigl(t_{1}^{2}\bigr) \otimes \cdots \otimes L\bigl(t_{\xi_2}^2 \bigr) \otimes { }^{\psi} V(x) $ and apply Proposition $\ref{w1}$ for $k=\xi_2$. Then~we~obtain
\begin{align}
&\bigl(\psi_2\bigl(\boldsymbol{t}^{2}\bigr)\bigl(\mathbb{B}_{\xi_3}^{\ltwo}\bigl(\boldsymbol{t}^{ 3}\bigr)\bigr)\bigr)^{\boldsymbol{b-2}} v\nonumber\\
&\qquad{}=\frac{1}{(\xi_{3}-|I|)!} \Symb_{ \boldsymbol{t}^{3}}\Biggl[F_I\bigl(\boldsymbol{t}^3,\boldsymbol{t}^2\bigr) \prod_{m=1}^{|I|} \frac{t^3_m-x+\Lambda^3}{t^3_m-x} \prod_{r=|I|+1}^{\xi_3} \frac{1}{t^3_r-x} \Biggr] e_{43}^{\xi_3-|I|}v, \label{step2}
\end{align}
where the subset $I\subset\{1,\dots,\xi_2\}$ and the sequence $\boldsymbol{b}=(b_1,\dots,b_{\xi_2})$ are related as follows: $b_j=3$ if $j\notin I$ and $b_j=4$ if $j\in I$. Therefore, by formula~\eqref{V} we have
\begin{equation*}
\bigl(\psi_2\bigl(\boldsymbol{t}^{2}\bigr)\bigl(\mathbb{B}_{\xi_3}^{\ltwo}\bigl(\boldsymbol{t}^{ 3}\bigr)\bigr)\bigr)^{\boldsymbol{b-2}} v= \prod_{r=1}^{\xi_3} \frac{1}{t^3_r-x} V_{I}\bigl(\boldsymbol{t}^{ 3}, \boldsymbol{t}^{ 2},x-\Lambda^3\bigr) e_{43}^{\xi_3-|I|}v.
\end{equation*}

The next step is to compute \smash{$\bigl(\phi_2\bigl(\boldsymbol{t}^{2}\bigr)\bigl(\mathbb{B}_{\xi_1}^{\ltwo}\bigl(\boldsymbol{t}^{ 1}\bigr)\bigr)\bigr)^{\boldsymbol{a}} e_{43}^{\xi_3-|I|}v$}. Notice that for any nonnegative integer $m$, we have
\begin{equation*}
T^2_1(u) e^{m}_{43}v=0, \qquad T^2_2(u) e^{m}_{43}v=\frac{u-x+\Lambda^2}{u-x}e^{m}_{43}v,\qquad T^1_1(u) e^{m}_{43}v=\frac{u-x+\Lambda^1}{u-x}e^{m}_{43}v.
\end{equation*}
Let ${ }^{\phi} V(x)$ be the $Y(\mathfrak{g l}_{2})$-module obtained by pulling back $V(x)$ through the embedding $\phi_2$.
To compute the \smash{$\bigl(\phi_2\bigl(\boldsymbol{t}^{2}\bigr)\bigl(\mathbb{B}_{\xi_1}^{\ltwo}\bigl(\boldsymbol{t}^{ 1}\bigr)\bigr)\bigr)^{\boldsymbol{a}} e_{43}^{\xi_3-|I|}v$}, we take the weight function \smash{$\mathbb{B}^{\ltwo}_{\xi_1}\bigl(\boldsymbol{t}^1\bigr)\bigl(e_{43}^{\xi_3-|I|}v\otimes \mathbf{w}_{2}^{\otimes \xi_2}\bigr)$} in the $Y(\mathfrak{g l}_{2})$-module \smash{${ }^{\phi} V(x)\otimes L\bigl(t_{\xi_2}^{2} \bigr) \otimes \cdots \otimes L\bigl(t_{1}^2\bigr) $} and apply Proposition $\ref{w2}$ for $k=\xi_2$. Then we obtain
\begin{gather}
\bigl(\phi_2\bigl(\boldsymbol{t}^{2}\bigr)\bigl(\mathbb{B}_{\xi_1}^{\ltwo}\bigl(\boldsymbol{t}^{ 1}\bigr)\bigr)\bigr)^{\boldsymbol{a}} e_{43}^{\xi_3-|I|}v\nonumber\\
\qquad{}=\frac{1}{(\xi_{1}-|J|)!} \Symb_{\boldsymbol{t}^{1}}\Biggl[\tilde{F}_J\bigl(\boldsymbol{t}^1,\boldsymbol{t}^2\bigr) \prod_{m=1}^{|J|} \frac{t^1_m-x+\Lambda^2}{t^1_m-x} \prod_{r=|J|+1}^{\xi_1} \frac{1}{t^1_r-x} \Biggr] e_{21}^{\xi_1-|J|}e_{43}^{\xi_3-|I|}v,\label{step3}
\end{gather}
where the subset $J\subset\{1,\dots,\xi_2\}$ and the sequence $\boldsymbol{a}=(a_1,\dots,a_{\xi_2})$ are related as follows: $a_j=1$ if $j\in J$ and $a_j=2$ if $j\notin J$. Therefore, by formula~\eqref{Vt} we have
\begin{equation*}
\bigl(\phi_2\bigl(\boldsymbol{t}^{2}\bigr)\bigl(\mathbb{B}_{\xi_1}^{\ltwo}\bigl(\boldsymbol{t}^{ 1}\bigr)\bigr)\bigr)^{\boldsymbol{a}} e_{43}^{\xi_3-|I|}v=\prod_{r=1}^{\xi_1} \frac{1}{t^1_r-x}\widetilde{V}_{J}\bigl(\boldsymbol{t}^{ 1}, \boldsymbol{t}^{ 2},x-\Lambda^2\bigr) e_{21}^{\xi_1-|J|}e_{43}^{\xi_3-|I|}v.
\end{equation*}

Finally, for the sequences $\boldsymbol{a}$, $\boldsymbol{b}$ that are related to the sets $I$, $J$ as above, we have
\begin{equation}\label{step4}
\Bigl(\overset{[2]}{\mathbb{T}} \bigl(\boldsymbol{t}^{2}\bigr)\Bigr)^{\boldsymbol{a}}_{\boldsymbol{b}} =\prod_{r=1}^{\xi_2} \frac{1}{t^2_r-x} e_{32}^{\xi_2-p-q+s} e_{31}^{q-s} e_{42}^{p-s} e_{41}^{s},
\end{equation}
where $p=|I|$, $q=|J|$, $s=|I\cap J|$. Now formula~\eqref{preth} follows from formulae~\eqref{step1}--\eqref{step4}.

\section{Proof of Proposition~\ref{symprop}}
\label{section6}
Consider the algebra $\mathcal{A}$ generated by two commuting copies of the symmetric group $S_k$ and rational functions of $z_1,\dots,z_k$ subject to relations~\eqref{relat} below. We denote the copies of~$S_k$ in~$\mathcal{A}$ by~\smash{$\dot S_k$} and~\smash{$\ddot S_k$}, and mark elements of \smash{$\dot S_k$} and \smash{$\ddot S_k$} by the corresponding dots, keeping the notation~$S_k$ without dots for the abstract symmetric group.

Let $\boldsymbol{z}=(z_1,\dots,z_k)$ and $\boldsymbol{z}^\sigma=(z_{\sigma(1)},\dots,z_{\sigma(k)})$. The additional relations in $\mathcal{A}$ are
\begin{equation}\label{relat}
\dot\sigma f(\boldsymbol{z})=f(\boldsymbol{z}^\sigma) \dot\sigma,\qquad \ddot\tau f(\boldsymbol{z})=f(\boldsymbol{z}) \ddot \tau.
\end{equation}

For $a=1,\dots,k-1$, let $s_a\in S_k$ be the transposition of $a$ and $a+1$. Consider the elements $\hat{s}_1,\dots,\hat{s}_{k-1}$ of $\mathcal{A}$,
\begin{equation}\label{shat}
\hat{s}_a=\biggl(\frac{z_a-z_{a+1}}{z_a-z_{a+1}-1}\ddot s_a-\frac{1}{z_a-z_{a+1}-1}\biggr)\dot s_a.
\end{equation}
It is straightforward to check that they satisfy the following relations:
\begin{equation*}
\hat{s}_a\hat{s}_{a+1}\hat{s}_a=\hat{s}_{a+1}\hat{s}_a\hat{s}_{a+1},\qquad \hat{s}_a^2=1.
\end{equation*}
Therefore, the assignment $s_a\mapsto \hat{s}_a$ defines an algebra homomorphism $\mathbb{C}S_k\rightarrow \mathcal{A}$. For any~${\sigma\in S_k}$, we denote by $\hat{\sigma}$ the corresponding element of $ \mathcal{A}$. Every element $\hat{\sigma}$ can be written in the following form:
\begin{equation}\label{sighat}
\hat{\sigma}=\sum_{\tau\in S_k} X_{\sigma,\tau}(\boldsymbol{z}) \ddot \tau \dot \sigma,
\end{equation}
where $X_{\sigma,\tau}(\boldsymbol{z})$ are functions of $z_1,\dots,z_k$.

Let $|\sigma|$ denote the length of $\sigma\in S_k$.
\begin{Lemma} \label{longperm} The functions $X_{\sigma,\tau}(\boldsymbol{z})$ have the following properties:
\begin{gather}\label{lm1}
X_{\sigma,\tau}(\boldsymbol{z})=0 \qquad \text{if} \ |\tau|>|\sigma|,
\\ \label{lm2}
X_{\sigma,\tau}(\boldsymbol{z})=\delta_{\sigma,\tau} X_{\sigma,\sigma}(\boldsymbol{z})
\qquad \text{if} \ |\tau|=|\sigma|,
\\ \label{lm3}
X_{\sigma,\sigma}(\boldsymbol{z})=\prod_{a<b,\,\sigma^{-1}(a)>\sigma^{-1}(b)}\frac{z_a-z_b}{z_a-z_b-1}.
\end{gather}
\end{Lemma}
\begin{proof}
Formulae~\eqref{lm1} and~\eqref{lm2} follow from formula~\eqref{shat} by inspection. Formula~\eqref{lm3} can be shown by induction on $|\sigma|$.
\end{proof}

Denote by $\sigma_0$ the longest element of $S_k$, $\sigma_0(i)=k-i+1$, $i=1,\dots,k$. Let
\begin{equation*}
\Phi(\boldsymbol{z})=\prod_{a<b}\frac{z_a-z_b-1}{z_a-z_b}.
\end{equation*}
Notice that
\begin{equation}\label{phix}
\Phi(\boldsymbol{z})=\frac{1}{X_{\sigma_0,\sigma_0}(\boldsymbol{z})}.
\end{equation}
\begin{Lemma}\label{deltalem}One has
\begin{equation} \label{deltalemf}
\sum_{\lambda\in S_k} X_{\lambda,\rho}(\boldsymbol{z})\Phi\bigl(\boldsymbol{z}^{\lambda\sigma_0}\bigr)X_{\sigma_0\lambda^{-1},\sigma_0\tau^{-1}}\bigl(\boldsymbol{z}^{\lambda\sigma_0}\bigr)=\delta_{\rho,\tau}.
\end{equation}
\end{Lemma}
\begin{proof}
Since $\widehat{\sigma\tau} =\hat{\sigma}\hat{\tau}$, by formula~\eqref{sighat} we have
\begin{equation}\label{prodx}
X_{\sigma\tau,\rho}(\boldsymbol{z})=\sum_\pi X_{\sigma,\pi}(\boldsymbol{z})X_{\tau,\pi^{-1}\rho}(\boldsymbol{z}^\sigma).
\end{equation}
Taking here $\rho=\sigma_0$, and using Lemma $\ref{longperm}$ and formula~\eqref{phix}, we get
\begin{equation}\label{Xform}
\sum_\pi X_{\sigma,\pi}(\boldsymbol{z})X_{\tau,\pi^{-1}\sigma_0}(\boldsymbol{z}^\sigma)=\delta_{\sigma\tau,\sigma_0} \frac{1}{\Phi(\boldsymbol{z})}.
\end{equation}
Replacing now $\boldsymbol{z}$ by \smash{$\boldsymbol{z}^{\sigma^{-1}}$} in formula~\eqref{Xform} and taking there $\tau=\mu^{-1}\sigma_0$, we get
\begin{equation}\label{xeq}
\sum_\pi X_{\sigma,\pi}\bigl(\boldsymbol{z}^{\sigma^{-1}}\bigr)X_{\mu^{-1}\sigma_0,\pi^{-1}\sigma_0}(\boldsymbol{z} )=\delta_{\sigma,\mu} \frac{1}{\Phi\bigl(\boldsymbol{z}^{\sigma^{-1}}\bigr)}.
\end{equation}
Formula~\eqref{xeq} can be understood as the matrix equality $AB=C$ for
$k!\times k!$ matrices~$A$,~$B$,~$C$ with entries labeled by permutations:
\begin{equation*}
A_{\sigma,\pi}=X_{\sigma,\pi}\bigl(\boldsymbol{z}^{\sigma^{-1}}\bigr),\qquad B_{\pi,\mu}=X_{\mu^{-1}\sigma_0,\pi^{-1}\sigma_0}(\boldsymbol{z}),\qquad C_{\sigma,\mu}= \delta_{\sigma,\mu} \frac{1}{\Phi\bigl(\boldsymbol{z}^{\sigma^{-1}}\bigr)}.
\end{equation*}
Therefore, the product $BC^{-1}A$ equals the identity matrix, which can be written as follows:
\begin{equation*}
\sum_{\mu} X_{\mu^{-1}\sigma_0,\pi^{-1}\sigma_0}(\boldsymbol{z})\Phi\bigl(\boldsymbol{z}^{\mu^{-1}}\bigr)X_{\mu,\sigma}\bigl(\boldsymbol{z}^{\mu^{-1}}\bigr)=\delta_{\pi,\sigma}.
\end{equation*}
After the substitution
$\lambda=\mu^{-1}\sigma_0$, $\rho=\pi^{-1}\sigma_0$, $\tau=\sigma^{-1}\sigma_0$, we get formula~\eqref{deltalemf}.
\end{proof}

\begin{Lemma}
One has
\begin{equation}\label{Xprop}
X_{\mu,\sigma}(\boldsymbol{z}^{s_a})=\frac{z_{a}-z_{a+1}}{z_{a}-z_{a+1}+1}X_{s_a\mu,s_a\sigma}(\boldsymbol{z} )+\frac{1}{z_{a}-z_{a+1}+1}X_{s_a\mu,\sigma}(\boldsymbol{z} ).
\end{equation}
\end{Lemma}
\begin{proof}By formulae~\eqref{shat} and~\eqref{sighat}, we have
\begin{equation*}
X_{s_a,s_a}(\boldsymbol{z})=\frac{z_a-z_{a+1}}{z_a-z_{a+1}-1}, \qquad X_{s_a,\id}(\boldsymbol{z})=\frac{-1}{z_a-z_{a+1}-1},
\end{equation*}
and $X_{s_a,\tau}(\boldsymbol{z})=0$, otherwise. Therefore, by formula~\eqref{prodx} we obtain
\begin{equation*}
X_{s_a\pi,\sigma}(\boldsymbol{z})=\frac{z_{a}-z_{a+1}}{z_{a}-z_{a+1}-1}X_{\pi,s_a\sigma}(\boldsymbol{z}^{s_a})-\frac{1}{z_a-z_{a+1}-1}X_{\pi,\sigma}(\boldsymbol{z}^{s_a}).
\end{equation*}
Replacing here $\boldsymbol{z}$ by $\boldsymbol{z}^{s_a}$ and making the substitution $\pi=s_a\mu$, we get formula~\eqref{Xprop}.
\end{proof}

\begin{Lemma} One has
\begin{equation}\label{Xprod2}
X_{\mu,\sigma }\bigl(\boldsymbol{z}^{s_a\mu^{-1}}\bigr)= \frac{z_{a}-z_{a+1}}{z_{a}-z_{a+1}-1} X_{\mu s_a,\sigma s_a} \bigl(\boldsymbol{z}^{s_a\mu^{-1}}\bigr)-\frac{1}{z_{a}-z_{a+1}-1}X_{\mu s_a,\sigma }\bigl(\boldsymbol{z}^{s_a\mu^{-1}}\bigr).
\end{equation}
\end{Lemma}
\begin{proof}
By formula~\eqref{prodx}, we have
\begin{equation*}
X_{\mu s_a,\sigma }(\boldsymbol{z})=\sum_{\pi} X_{\mu,\pi} (\boldsymbol{z})X_{s_a,\pi^{-1}\sigma } (\boldsymbol{z}^{\mu}).
\end{equation*}
Thus
\begin{equation*}
X_{\mu s_a,\sigma }(\boldsymbol{z})=\sum_{\pi} X_{\mu,\pi} (\boldsymbol{z})X_{s_a,\pi^{-1}\sigma } (\boldsymbol{z}^{\mu})=X_{\mu,\sigma } (\boldsymbol{z})X_{s_a, {\rm id}} (\boldsymbol{z}^{\mu})+X_{\mu,\sigma s_a} (\boldsymbol{z})X_{s_a, s_a} (\boldsymbol{z}^{\mu}).
\end{equation*}
Replacing here $\boldsymbol{z}$ by $\boldsymbol{z}^{\mu^{-1}}$, we get
\begin{equation*}
X_{\mu s_a,\sigma }\bigl(\boldsymbol{z}^{\mu^{-1}}\bigr)=X_{\mu,\sigma } \bigl(\boldsymbol{z}^{\mu^{-1}}\bigr)X_{s_a,{\rm id}} (\boldsymbol{z} )+X_{\mu,\sigma s_a} \bigl(\boldsymbol{z}^{\mu^{-1}}\bigr)X_{s_a, s_a} (\boldsymbol{z} ).
\end{equation*}
Substituting now $\mu$ with $\mu s_a$, we obtain~\eqref{Xprod2}.
\end{proof}

For $\sigma\in S_k$ and a subset $I=\{i_1,\dots,i_m\}\subset\{1,\dots,k\}$, denote
\begin{equation*}\sigma (I)=\{\sigma(i_1),\dots,\sigma(i_m)\}.\end{equation*}
Recall the functions $V_I(\boldsymbol{t},\boldsymbol{z},y)$, $\widetilde{V}_J(\boldsymbol{t},\boldsymbol{z},y)$, see formulae~\eqref{F},~\eqref{V} and~\eqref{Ft},~\eqref{Vt}.

\begin{Lemma}
For each $a=1,\dots, k-1$, we have
\begin{gather}\label{sym2}
V_{I}(\boldsymbol{t},\boldsymbol{z}^{s_a}\!,y)=\frac{z_{a+1}-z_a}{z_{a+1}-z_{a}-1}V_{s_a(I)}(\boldsymbol{t},\boldsymbol{z},y)-\frac{1}{z_{a+1}-z_{a}-1}V_{I}(\boldsymbol{t},\boldsymbol{z},y),
\\ \label{sym1}
\widetilde{V}_I(\boldsymbol{t},\boldsymbol{z}^{s_a}\!,y)=\frac{z_a-z_{a+1}}{z_a-z_{a+1}-1}\widetilde{V}_{s_a(I)}(\boldsymbol{t},\boldsymbol{z},y)-\frac{1}{z_a-z_{a+1}-1}\widetilde{V}_I(\boldsymbol{t},\boldsymbol{z},y).
\end{gather}
\end{Lemma}

\begin{proof}
By the structure of formulae~\eqref{F} and~\eqref{V} for the function $V_I(\boldsymbol{t},\boldsymbol{z},y)$, it is enough to prove formula~\eqref{sym2} for $k=2$. In this case, the statement follows from the identities
\begin{gather*}
1=\frac{z'-z}{z'-z-1}-\frac{1}{z'-z-1},
\\
\frac{1}{t-z'}\cdot\frac{t-z+1}{t-z}=\frac{z'-z}{z'-z-1}\cdot\frac{1}{t-z'}-\frac{1}{z'-z-1}\cdot\frac{1}{t-z}\cdot\frac{t-z'+1}{t-z'},
\\
\bigl(t'-z\bigr)\bigl(t-z'+1\bigr)\bigl(t-t'-1\bigr)-(t-z)\bigl(t'-z'+1\bigr)\bigl(t'-t-1\bigr)\\
\qquad{}=\bigl(t'-z'\bigr)(t-z+1)\bigl(t-t'-1\bigr)-\bigl(t-z'\bigr)\bigl(t'-z+1\bigr)\bigl(t'-t-1\bigr).
\end{gather*}

The proof of~\eqref{sym1} is similar by using formulae~\eqref{Ft} and~\eqref{Vt} for functions $\widetilde{V}_J(\boldsymbol{t},\boldsymbol{z},y)$.
\end{proof}

The statement of Proposition $\ref{symprop}$ is given by formula~\eqref{sympropf}. It can be written as follows:
\begin{align}
&\sum\limits_{(I,J)\in \mathcal{S}_{p,q,r,k}}V_{I}\bigl(\boldsymbol{t}^3, \boldsymbol{z},x-\Lambda^3\bigr) \widetilde{V}_{J}\bigl( \boldsymbol{t}^1, \boldsymbol{z},x-\Lambda^2\bigr) \nonumber\\
&\qquad{}=\frac{1}{C_{p,q,r,k}}\sum_{\sigma\in S_k} V_{\sigma_0({I}_0)}\bigl(\boldsymbol{t}^3, \boldsymbol{z}^{\sigma},x-\Lambda^3\bigr) \widetilde{V}_{J_0}\bigl( \boldsymbol{t}^1, \boldsymbol{z}^{\sigma\sigma_0},x-\Lambda^2\bigr) \Phi(\boldsymbol{z}^\sigma),\label{new54}
\end{align}
where $\mathcal S_{p,q,r,k}$ on the left-hand side is the set of all pairs of subsets $I$, $J$ of $\{1,\dots,k \}$, such that $|I|=p$, $|J|=q$, $|I\cap J|=r$, and we use $k=\xi_2$, $\boldsymbol{z}=\boldsymbol{t}^2$. On the right-hand side,
\[
C_{p,q,r,k}= (p-r)!(q-r)! r!(k-p-q+r)!
\]
 and $(I_0,J_0)$  is any fixed pair from $\mathcal S_{p,q,r,k}$. We also expanded $\Symb_{\boldsymbol{t}^2}$ according to formulae~\eqref{Qfun} and~\eqref{barsym}, and observed that
 \[
 \check{I}_0=\sigma_0(I_0) , \qquad \check{\boldsymbol{t}}^2=\boldsymbol{z}^{\sigma_0}.
 \]

In the rest of the proof, we will suppress the arguments $\boldsymbol{t}^1$, $\boldsymbol{t}^3$, $x-\Lambda^2$, $x-\Lambda^3$ because they are the same on both sides of formula~\eqref{new54} and will never be changed in the reasoning.

Notice that every pair $(I,J)\in \mathcal{S}_{p,q,r,k}$ can be obtained from an arbitrary fixed pair $(I_0,J_0) \in \mathcal{S}_{p,q,r,k}$ by the action of the symmetric group $S_k$. Therefore, the left-hand side of the formula~\eqref{new54} can be written in the following way:
\begin{equation}\label{twosums}
\sum_{(I,J)\in \mathcal{S}_{p,q,r,k}}\widetilde{V}_J( \boldsymbol{z}) V_I(\boldsymbol{z})=\frac{1}{C_{p,q,r,k}}\sum_{\sigma\in S_k} \widetilde{V}_{\sigma(J_0)}(\boldsymbol{z}) V_{\sigma(I_0)}(\boldsymbol{z} ).
\end{equation}

Using Lemma $\ref{deltalem}$, we get{\samepage
\begin{gather}
\sum_{\sigma\in S_k} \widetilde{V}_{\sigma(J_0)}(\boldsymbol{z}) V_{\sigma(I_0)}(\boldsymbol{z} ) \nonumber\\
\qquad{}=\sum_{\sigma,\pi,\tau\in S_k} \widetilde{V}_{\sigma(J_0)}(\boldsymbol{z})X_{\pi,\sigma}(\boldsymbol{z}) \Phi(\boldsymbol{z}^{\pi\sigma_0})X_{\sigma_0\pi^{-1},\sigma_0\tau^{-1}}(\boldsymbol{z}^{\pi\sigma_0})V_{\tau(I_0)}(\boldsymbol{z}),\label{longeq}
\end{gather}}%
since from formula~\eqref{deltalemf}
\begin{equation*}
\sum_{\pi\in S_k} X_{\pi,\sigma}(\boldsymbol{z}) \Phi(\boldsymbol{z}^{\pi\sigma_0})X_{\sigma_0\pi^{-1},\sigma_0\tau^{-1}}(\boldsymbol{z}^{\pi\sigma_0}) =\delta_{\sigma,\tau}.
\end{equation*}
\begin{Lemma}\label{sum1} We have
\begin{equation}\label{sumvx}
\sum_{\sigma\in S_k} \widetilde{V}_{\sigma(J_0)}(\boldsymbol{z})X_{\pi,\sigma}(\boldsymbol{z})=\widetilde{V}_{J_0}(\boldsymbol{z}^\pi).
\end{equation}
\end{Lemma}
\begin{proof}
We will use induction on the length of the permutation $\pi$. For $\pi=\id$, formula~\eqref{sumvx} is clear, and for $\pi=s_a$ with some $a=1,\dots,k-1$, formula~\eqref{sumvx} coincides with formula~\eqref{sym1}. For the induction step, we find $a$ such that $|s_a\pi|=|\pi|-1$, and denote $\rho =s_a\pi$. Then by the induction assumption
\begin{equation*}
\sum_{\sigma} \widetilde{V}_{\sigma(J_0)}(\boldsymbol{z})X_{\rho,\sigma}(\boldsymbol{z})= \widetilde{V}_{J_0}(\boldsymbol{z}^\rho).
\end{equation*}
Replacing here $\boldsymbol{z}$ by $\boldsymbol{z}^{s_a}$, we get
\begin{equation}\label{indstep}
\sum_{\sigma} \widetilde{V}_{\sigma(J_0)}(\boldsymbol{z}^{s_a})X_{\rho,\sigma}(\boldsymbol{z}^{s_a})= \widetilde{V}_{J_0}(\boldsymbol{z}^{s_a\rho})=\widetilde{V}_{J_0}(\boldsymbol{z}^{\pi}).
\end{equation}
Using formulae~\eqref{Xprop} and~\eqref{sym1}, the left-hand side of~\eqref{indstep} becomes
\begin{gather*}
\sum_{\sigma} \frac{(z_a-z_{a+1})^2}{(z_a-z_{a+1})^2-1}\widetilde{V}_{s_a\sigma(J_0)}(\boldsymbol{z})X_{s_a\rho,s_a\sigma}(\boldsymbol{z})\\
\qquad{} +\sum_{\sigma}\frac{z_a-z_{a+1}}{(z_a-z_{a+1})^2-1}\widetilde{V}_{s_a\sigma(J_0)}(\boldsymbol{z})X_{s_a\rho,\sigma}(\boldsymbol{z})\\
\qquad{}- \sum_{\sigma}\frac{z_a-z_{a+1}}{(z_a-z_{a+1})^2-1}\widetilde{V}_{\sigma(J_0)}(\boldsymbol{z})X_{s_a\rho,s_a\sigma}(\boldsymbol{z})\\
\qquad{} -\sum_{\sigma} \frac{1}{(z_a-z_{a+1})^2-1}\widetilde{V}_{\sigma(J_0)}(\boldsymbol{z})X_{s_a\rho, \sigma}(\boldsymbol{z}).
\end{gather*}
Changing the summation index in the first and second sums from $\sigma$ to $s_a\sigma$, we observe that the second and third sums cancel each other, while the first and fourth sums combine together and simplify to the expression
\begin{equation*}
\sum_{\sigma} \widetilde{V}_{\sigma(J_0)}(\boldsymbol{z})X_{s_a\rho,\sigma}(\boldsymbol{z}) =\sum_{\sigma} \widetilde{V}_{\sigma(J_0)}(\boldsymbol{z})X_{\pi,\sigma}(\boldsymbol{z}),
\end{equation*}
which appears on the left-hand side of formula~\eqref{sumvx}.
\end{proof}

\begin{Lemma}\label{sum2} We have
\begin{equation}\label{sum2f}
\sum_{\tau}V_{\tau(I_0)}(\boldsymbol{z})X_{\sigma_0\pi^{-1},\sigma_0\tau^{-1}}(\boldsymbol{z}^{\pi\sigma_0})=V_{\sigma_0(I_0)}(\boldsymbol{z}^{\pi\sigma_0}).
\end{equation}
\end{Lemma}
\begin{proof}
Recall the notation $\sigma_0(I_0)=\check{I_0}$. Transform formula~\eqref{sum2f} by making the substitutions $\mu=\sigma_0\pi^{-1}$, $\sigma=\sigma_0\tau^{-1}$,{\samepage
\begin{equation}\label{sum2fn}
\sum_{\sigma}V_{\sigma^{-1}(\check{I_0})}(\boldsymbol{z})X_{\mu, \sigma }\bigl(\boldsymbol{z}^{\mu^{-1}}\bigr)=V_{\check{I_0}}\bigl(\boldsymbol{z}^{\mu^{-1}}\bigr).
\end{equation}
The rest of the proof is analogous to that of Lemma $\ref{sum1}$.}

To prove formula~\eqref{sum2fn}, we will use induction on the length of $\mu$. For $\mu=\id$, formula~\eqref{sum2fn} is clear, and for $\mu=s_a$ with some $a=1,\dots,k-1$, formula~\eqref{sum2fn} coincides with formula~\eqref{sym2}. For the induction step, we find $a$ such that $|\mu s_a|=|\mu|-1$, and denote $\rho=\mu s_a$. Then by the induction assumption,
\begin{equation*}
\sum_{\sigma} V_{\sigma^{-1}(\check{I}_0)}(\boldsymbol{z})X_{\rho,\sigma }\bigl(\boldsymbol{z}^{\rho^{-1}}\bigr) =V_{\check{I}_0}\bigl(\boldsymbol{z}^{\rho^{-1}}\bigr),
\end{equation*}
and replacing here $\boldsymbol{z}$ by $\boldsymbol{z}^{s_a}$, we get
\begin{equation}\label{indstep2}
\sum_{\sigma} V_{\sigma^{-1}(\check{I}_0)}(\boldsymbol{z}^{s_a})X_{\rho,\sigma }\bigl(\boldsymbol{z}^{s_a\rho^{-1}}\bigr)=V_{\check{I}_0}\bigl(\boldsymbol{z}^{\mu^{-1}}\bigr)= V_{\check{I}_0}\bigl(\boldsymbol{z}^{s_a\rho ^{-1}}\bigr).
\end{equation}
Using formulae~\eqref{Xprod2},~\eqref{sym2}, the left-hand side of~\eqref{indstep2} becomes
\begin{gather*}
\sum_{\sigma} \frac{(z_{a+1}-z_a)^2}{(z_{a+1}-z_{a})^2-1}V_{s_a\sigma^{-1}(\check{I}_0)}(\boldsymbol{z})X_{\rho s_a,\sigma s_a}\bigl(\boldsymbol{z}^{s_a\rho^{-1}}\bigr) \\
\qquad{}-\sum_{\sigma}\frac{z_a-z_{a+1}}{(z_a-z_{a+1})^2-1}V_{s_a\sigma^{-1}(\check{I}_0)}(\boldsymbol{z})X_{\rho s_a,\sigma }\bigl(\boldsymbol{z}^{s_a\rho^{-1}}\bigr)
\\
\qquad{}+\sum_{\sigma} \frac{z_a-z_{a+1}}{(z_a-z_{a+1})^2-1}V_{\sigma^{-1}(\check{I}_0)}(\boldsymbol{z})X_{\rho s_a,\sigma s_a}\bigl(\boldsymbol{z}^{s_a\rho^{-1}}\bigr) \\
\qquad{}-\sum_{\sigma} \frac{1}{(z_{a+1}-z_{a})^2-1}V_{\sigma^{-1}(\check{I}_0)}(\boldsymbol{z})X_{\rho s_a, \sigma }\bigl(\boldsymbol{z}^{s_a\rho^{-1}}\bigr).
\end{gather*}
Changing the summation index in the first and second sums from $\sigma$ to $\sigma s_a$, we observe that the second and third sums cancel each other, while the first and the fourth sums combine together and simplify to the expression
\begin{equation*}
\sum_{\sigma} V_{\sigma^{-1}(\check{I}_0)}(\boldsymbol{z})X_{\rho s_a,\sigma }\bigl(\boldsymbol{z}^{s_a\rho^{-1}}\bigr) =\sum_{\sigma} V_{\sigma^{-1}(\check{I}_0)}(\boldsymbol{z})X_{\mu,\sigma }\bigl(\boldsymbol{z}^{\mu^{-1}}\bigr),
\end{equation*}
which appears on the left-hand side of formula~\eqref{sum2fn}.
\end{proof}

Using Lemmas $\ref{sum1}$ and $\ref{sum2}$, we evaluate the sums over $\sigma$ and $\tau$ on the right-hand side of formula~\eqref{longeq} and get the equality
\begin{equation}\label{keypoint}
\sum_{\sigma\in S_k} \widetilde{V}_{\sigma(J_0)}(\boldsymbol{z}) V_{\sigma(I_0)}(\boldsymbol{z} )=\sum_{\pi} V_{\sigma_0(I_0)}(\boldsymbol{z}^{\pi\sigma_0})\widetilde{V}_{J_0}(\boldsymbol{z}^{\pi})\Phi(\boldsymbol{z}^{\pi\sigma_0}).
\end{equation}
Using formula~\eqref{twosums} on the left-hand side and making the substitution $\pi=\sigma\sigma_0$ on the right-hand side, we obtain that $\eqref{keypoint}$ can be written as
\begin{equation*}
\sum_{(I,J)\in \mathcal{S}_{p,q,r,k}}\widetilde{V}_J( \boldsymbol{z}) V_I(\boldsymbol{z})=\frac{1}{C_{p,q,r,k}} \sum_{\sigma} V_{\sigma_0(I_0)}(\boldsymbol{z}^{\sigma})\widetilde{V}_{J_0}(\boldsymbol{z}^{\sigma\sigma_0})\Phi(\boldsymbol{z}^{ \sigma}),
\end{equation*}
which is formula~\eqref{new54}. Proposition~\ref{symprop} is proved.

\appendix

\section{Proof of Proposition~\ref{twotensor}}\label{appendixA}

In this appendix, we will consider only the algebra $Y(\mathfrak{gl}_2)$ and, for convenience, we will not write the superscript $\ltwo$.
We will use the commutation relations
\begin{gather}
T_{11} (u)T_{11} (t)=T_{11} (t)T_{11} (u),\qquad
 T_{12} (u) T_{12} (t)= T_{12} (t) T_{12} (u),\nonumber\\
 T_{22} (u) T_{22} (t)= T_{22} (t) T_{22} (u),\label{2com1}
\\ \label{2com2}
T_{11} (u) T_{12} (t)=\frac{u-t-1}{u-t} T_{12} (t) T_{11} (u) +\frac{1}{u-t} T_{12} (u) T_{11} (t),
\\ \label{2com3}
T_{22} (u) T_{12} (t)=\frac{u-t+1}{u-t} T_{12} (t) T_{22} (u)-\frac{1}{u-t} T_{12} (u) T_{22} (t),
\end{gather}
following from the defining relations in $Y(\mathfrak{g l}_{2})$, see~\eqref{trel}. We will also use the next statement.
\begin{Proposition}
One has
\begin{gather}
T_{11} (u) T_{12} (t_1) \cdots T_{12} (t_k)=\prod_{i=1}^k\frac{u-t_i-1}{u-t_i} T_{12} (t_1)\cdots T_{12} (t_k) T_{11} (u)\label{ABf}\\[-1mm]
\qquad{}+\sum_{l=1}^k \frac{1}{u-t_l}\prod_{\substack{m=1\\m\neq l}}^k \frac{t_l-t_m-1}{t_l-t_m} T_{12} (t_1)\cdots T_{12} (t_{l-1}) T_{12} (t_{l+1})\cdots T_{12} (t_k) T_{12} (u)T_{11} (t_l),\nonumber
\\[-1mm]
T_{22} (u) T_{12} (t_1) \cdots T_{12} (t_k)=\prod_{i=1}^k\frac{u-t_i+1}{u-t_i} T_{12} (t_1)\cdots T_{12} (t_k) T_{22} (u)\label{DBf}\\[-1mm]
\qquad{} -\sum_{l=1}^k \frac{1}{u-t_l}\prod_{\substack{m=1\\m\neq l}}^k \frac{t_l-t_m+1}{t_l-t_m} T_{12} (t_1)\cdots T_{12} (t_{l-1}) T_{12} (t_{l+1})\cdots T_{12} (t_k) T_{12} (u)T_{22} (t_l).\nonumber
\end{gather}
\end{Proposition}
\begin{proof}
The statement goes back to \cite{FT}. We will prove it by induction on $k$. Consider formula~\eqref{ABf}.
The statement for $k=1$ is given by formula~\eqref{2com2}. We use the induction assumption to move $T_{11}(u)$ through the product $T_{12} (t_1) \cdots T_{12} (t_{k-1})$:
\begin{gather} \label{ABi}
T_{11} (u) T_{12} (t_1) \cdots T_{12} (t_{k-1}) T_{12} (t_{k})=\prod_{i=1}^{k-1}\frac{u-t_i-1}{u-t_i} T_{12} (t_1)\cdots T_{12} (t_{k-1}) T_{11} (u)T_{12} (t_{k})\\[-1mm]
+\sum_{l=1}^{k-1} \frac{1}{u-t_l}\prod_{\substack{m=1\\m\neq l}}^{k-1} \frac{t_l-t_m-1}{t_l-t_m} T_{12} (t_1)\cdots T_{12} (t_{l-1}) T_{12} (t_{l+1})\cdots T_{12} (t_{k-1}) T_{12} (u)T_{11} (t_l)T_{12} (t_{k}).\nonumber
\end{gather}
Then we apply~\eqref{2com2} to the product $ T_{11} (u)T_{12} (t_{k})$ and $T_{11} (t_l)T_{12} (t_{k})$ and the right-hand side of~\eqref{ABi} becomes
\begin{gather*}
 \prod_{i=1}^{k}\frac{u-t_i-1}{u-t_i} T_{12} (t_1)\cdots T_{12} (t_{k-1}) T_{12}(t_{k}) T_{11} (u)\\[-1mm]
\quad{}+\frac{1}{u-t_{k}} \prod_{i=1}^{k-1}\frac{u-t_i-1}{u-t_i} T_{12} (t_1)\cdots T_{12} (t_{k-1}) T_{12}(u) T_{11}(t_k) \\[-1mm]
\quad{}+ \sum_{l=1}^{k-1} \frac{1}{u-t_l}\prod_{\substack{m=1\\m\neq l}}^k \frac{t_l-t_m-1}{t_l-t_m} \\[-1mm]
\quad\hphantom{+ \sum_{l=1}^{k-1} \frac{1}{u-t_l}\prod_{\substack{m=1\\m\neq l}}^k}{}\, \times T_{12} (t_1)\cdots T_{12} (t_{l-1}) T_{12} (t_{l+1})\cdots T_{12} (t_{k-1})T_{12} (u) T_{12} (t_{k})T_{11} (t_l)\\
\quad{} +\sum_{l=1}^{k-1} \frac{1}{u-t_l}\frac{1}{t_l-t_k}\prod_{\substack{m=1\\m\neq l}}^{k-1} \frac{t_l-t_m-1}{t_l-t_m} \\
\quad\hphantom{+\sum_{l=1}^{k-1} \frac{1}{u-t_l}\frac{1}{t_l-t_k}\prod_{\substack{m=1\\m\neq l}}^{k-1}}{}\, \times T_{12} (t_1)\cdots T_{12} (t_{l-1}) T_{12} (t_{l+1})\cdots T_{12} (t_{k-1})T_{12} (u) T_{12} (t_l)T_{11} (t_k).
\end{gather*}
The first term here coincides with the first term on the right-hand side of formula~\eqref{ABf}. The third term here is the second term of~\eqref{ABf} without $l=k$ summand. We also used that $T_{12}(u)$ and $T_{12}(t_k)$ commute, see~\eqref{2com1}. The second and fourth summands combine into the product
\begin{equation}\label{l=m}
\frac{1}{u-t_k}\prod_{m=1}^{k-1} \frac{t_k-t_m-1}{t_k-t_m} T_{12} (t_1)\cdots T_{12} (t_{k-1}) T_{12} (u)T_{11} (t_k),
\end{equation}
using the following identity:
\begin{gather*}
\frac{1}{u-t_k} \prod_{i=1}^{k-1}\frac{u-t_i-1}{u-t_i} + \sum_{l=1}^{k-1} \frac{1}{(u-t_l)(t_l-t_k)}\prod_{\substack{m=1\\m\neq l}}^{k-1} \frac{t_l-t_m-1}{t_l-t_m}= \frac{1}{u-t_k}\prod_{m=1}^{k-1} \frac{t_k-t_m-1}{t_k-t_m}.
\end{gather*}
The product~\eqref{l=m} is exactly the summand with $l=k$ of the second term in~\eqref{ABf}. Formula~\eqref{ABf} is proved.

The proof of formula~\eqref{DBf} is similar to that of formula~\eqref{ABf} with relation~\eqref{2com3} used instead of~\eqref{2com2}.
\end{proof}

Recall that for the $\mathfrak{gl}_2$ case we have
\begin{equation*}
{\mathbb{B}}_\xi(\boldsymbol{t})=T_{12} (t_{1} ) \cdots T_{12} (t_{\xi} ),
\end{equation*}
and thus Proposition~\ref{twotensor} can be rewritten as follows.
\begin{Proposition} \label{twotensorsA}
Let $\xi$ be a nonnegative integer and $\boldsymbol{t}=$ $(t_{1}, \dots, t_{\xi} )$. Then
\begin{gather}
\Delta ( T_{12} (t_{1} ) \cdots T_{12} (t_{\xi} ) )
\label{cop2}\\
\qquad{}=\sum_{\eta=0}^{\xi} \frac{1}{(\xi-\eta) ! \eta!} \Symb_{\boldsymbol{t}}\Biggl[ \Biggl(\prod_{i=1}^{\eta} T_{12} (t_i) \otimes \prod_{j=\eta+1}^{\xi } T_{12} (t_j) \Biggr)
\Biggl(\prod_{k=\eta+1}^{\xi } T _{22}(t_{k})\otimes \prod_{l=1}^{\eta} T _{11}(t_l)\Biggr) \Biggr].\nonumber
\end{gather}

\end{Proposition}
\begin{Remark}
Notice that according to~\eqref{2com1}, the factors in each of the large products commute among themselves, so the order of the factors is irrelevant.
\end{Remark}
\begin{proof}
Consider the summand from the right-hand side of~\eqref{cop2} with a given $\eta$,
\begin{equation} \label{start}
F_{\eta,\xi-\eta}(\boldsymbol{t})= \Symb_{\boldsymbol{t}}\Biggl[ \Biggl(\prod_{i=1}^{\eta} T_{12} (t_i) \otimes \prod_{j=\eta+1}^{\xi } T_{12} (t_j) \Biggr)
\Biggl(\prod_{k=\eta+1}^{\xi } T _{22}(t_{k})\otimes \prod_{l=1}^{\eta} T _{11}(t_l)\Biggr) \Biggr].
\end{equation}
Let
\begin{align*}
&P_{\eta,\xi-\eta} (\boldsymbol{t})= \prod_{1\leq i\leq\eta<j \leq \xi} \frac{t_i-t_j-1}{t_i-t_j}\\
& \hphantom{P_{\eta,\xi-\eta} (\boldsymbol{t})=}{}\times{}\Biggl(\prod_{i=1}^{\eta} T_{12} (t_i) \otimes \prod_{j=\eta+1}^{\xi } T_{12} (t_j) \Biggr) \Biggl(\prod_{k=\eta+1}^{\xi } T _{22}(t_{k})\otimes \prod_{l=1}^{\eta} T _{11}(t_l)\Biggr),\\
& U_{\eta,\xi-\eta}(\boldsymbol{t})= \prod_{1\leq i<j\leq \eta} \frac{t_{i}-t_{j}-1}{t_{i}-t_{j}} \prod_{\eta+1\leq i<j\leq \xi} \frac{t_{i}-t_{j}-1}{t_{i}-t_{j}}, \qquad
 \boldsymbol{t}^\sigma=\bigl(t_{\sigma(1)},\dots,t_{\sigma(\xi)}\bigr).
\end{align*}
Using this notation, formula~\eqref{start} can be written as
\begin{equation*}
F_{\eta,\xi-\eta}(\boldsymbol{t})=\sum_{\sigma\in S_\xi}\ U_{\eta,\xi-\eta}(\boldsymbol{t}^\sigma) P_{\eta,\xi-\eta}(\boldsymbol{t}^\sigma).
\end{equation*}
Observe that $ F_{\eta,\xi-\eta}(\boldsymbol{t})$ is symmetric in $t_1,\dots,t_\xi$.
Denote by $S_\eta \times S_{\xi-\eta} $ the subgroup of $S_\xi$ stabilizing the subsets $\{1,\dots,\eta\}$ and $\{\eta+1,\dots,\xi\}$. We have
\begin{align*}
\begin{split}
F_{\eta,\xi-\eta}(\boldsymbol{t})&{}=\frac{1}{\eta!(\xi-\eta)!} \sum_{\tau\in S_\eta\times S_{\xi-\eta}} F_{\eta,\xi-\eta}(\boldsymbol{t}^\tau)\\
&{}=\frac{1}{\eta!(\xi-\eta)!} \sum_{\tau\in S_\eta\times S_{\xi-\eta}} \sum_{\sigma\in S_\xi}\ U_{\eta,\xi-\eta}(\boldsymbol{t}^{\tau\sigma}) P_{\eta,\xi-\eta}(\boldsymbol{t}^{\tau\sigma}).
\end{split}
\end{align*}
Changing the summation variable in the inner sum, $\sigma=\tau^{-1}\rho\tau$, and using the fact that $P_{\eta,\xi-\eta}(\boldsymbol{t}^{\rho\tau})=P_{\eta,\xi-\eta}(\boldsymbol{t}^{\rho})$ for all $\tau \in S_\eta\times S_{\xi-\eta}$, we get
\begin{align*}
F_{\eta,\xi-\eta}(\boldsymbol{t})&{}=\frac{1}{\eta!(\xi-\eta)!} \sum_{\tau\in S_\eta\times S_{\xi-\eta}} \sum_{\rho\in S_\xi} U_{\eta,\xi-\eta}(\boldsymbol{t}^{\rho\tau}) P_{\eta,\xi-\eta}(\boldsymbol{t}^{\rho\tau})\\
&{}=\frac{1}{\eta!(\xi-\eta)!} \sum_{\rho\in S_\xi} P_{\eta,\xi-\eta}(\boldsymbol{t}^{\rho}) \sum_{\tau\in S_\eta\times S_{\xi-\eta}} U_{\eta,\xi-\eta}(\boldsymbol{t}^{\rho\tau}).
\end{align*}
Furthermore, using the identity
\begin{equation}
\label{n!}
\sum_{\tau\in S_n} \prod_{1\leq i<j\leq n} \frac{x_{\tau(i)}-x_{\tau(j)}-1}{x_{\tau(i)}-x_{\tau(j)}}=n!,
\end{equation}
we obtain that
$
\sum_{\tau\in S_\eta\times S_{\xi-\eta}}U_{\eta,\xi-\eta}(\boldsymbol{t}^{\rho\tau})=\eta!(\xi-\eta)!$
 and
\begin{equation}\label{formF}
F_{\eta,\xi-\eta}(\boldsymbol{t})=\sum_{\rho\in S_\xi} P_{\eta,\xi-\eta}(\boldsymbol{t}^\rho).
\end{equation}

Using formula~\eqref{formF}, the statement of Proposition~\ref{twotensorsA} can be formulated as follows:
\begin{equation}\label{state}
\Delta ( T_{12} (t_{1} ) \cdots T_{12} (t_{\xi} ) )=\sum_{\eta=0}^\xi \frac{1}{(\xi-\eta)!\eta!}\sum_{\rho\in S_\xi} P_{\eta,\xi-\eta}(\boldsymbol{t}^\rho).
\end{equation}
We prove this formula by induction on $\xi$. The base case $\xi=1$ is given by formula~\eqref{coproduct}:
\begin{equation}\label{gl2cop}
\Delta(T_{12} (t_1))=T_{12} (t_1)\otimes T_{11} (t_1)+T_{22} (t_1)\otimes T_{12} (t_1).
\end{equation}
To make the induction step, we use that
\begin{equation}\label{RHS}
\Delta ( {\mathbb{B}}_\xi(\boldsymbol{t}) )=\Delta\bigl(T_{12} (t_1)\bigr) \Delta (T_{12} (t_2)\cdots T_{12} (t_\xi)),
\end{equation}
expand the first factor according to~\eqref{gl2cop}, and apply the induction assumption to expand the second factor. Denote by $S_{\xi-1}'\subset S_\xi$ the subgroup of permutations $\rho$, such that $\rho(1)=1$. Then the right-hand side of formula~\eqref{RHS} becomes{\samepage
\begin{gather}
T_{12} (t_1)\otimes T_{11} (t_1)\sum_{\eta=1}^{\xi}\frac{1}{(\xi-\eta)!(\eta-1)!} \sum_{\tau\in S_{\xi-1}'} P_{\eta-1,\xi-\eta}\bigl(t_{\tau(2)},\dots,t_{\tau(\xi)}\bigr)\nonumber\\
\qquad{}+T_{22} (t_1)\otimes T_{12} (t_1) \sum_{\eta=0}^{\xi-1}\frac{1}{(\xi-1-\eta)!\eta!} \sum_{\rho\in S_{\xi-1}'} P_{\eta,\xi-\eta-1}\bigl(t_{\rho(2)},\dots,t_{\rho(\xi)}\bigr),\label{TTF}
\end{gather}}%

\pagebreak

\noindent
where in the first term we shifted the summation variable of the exterior sum.
Using the definition of $P_{\eta,\xi-1-\eta}(\boldsymbol{t}) $ and $P_{\eta-1,\xi-\eta}(\boldsymbol{t}) $, we further expand expression~\eqref{TTF}:
\begin{gather*}
\sum_{\eta=1}^{\xi} \frac{1}{(\xi-\eta)!(\eta-1)! } \sum_{\tau\in S_{\xi-1}'} \prod_{1< i\leq\eta<j \leq \xi} \frac{t_{\tau(i)}-t_{\tau(j)}-1}{t_{\tau(i)}-t_{\tau(j)}} \\[1mm]
\quad\times \Biggl(T_{12}(t_1) \prod_{i=2}^{\eta} T_{12} \bigl(t_{\tau(i)}\bigr) \otimes T_{11}(t_1)\prod_{j=\eta+1}^{\xi} T_{12} \bigl(t_{\tau(j)}\bigr) \Biggr) \Biggl(\prod_{k=\eta+1}^{\xi } T _{22}\bigl(t_{\tau(k)}\bigr)\otimes \prod_{l=2}^{\eta} T _{11}\bigl(t_{\tau(l)}\bigr)\Biggr) \\[1mm]
 + \sum_{\eta=0}^{\xi-1} \frac{1}{(\xi-\eta-1) ! \eta!} \sum_{\rho\in S_{\xi-1}'} \prod_{1< i\leq\eta+1<j \leq \xi} \frac{t_{\rho(i)}-t_{\rho(j)}-1}{t_{\rho(i)}-t_{\rho(j)}} \\[1mm]
\quad\times \Biggl(T_{22}(t_1) \prod_{i=2}^{\eta+1} T_{12} \bigl(t_{\rho(i)}\bigr) \otimes T_{12}(t_1)\prod_{j=\eta+2}^{\xi } T_{12} \bigl(t_{\rho(j)}\bigr) \Biggr)
\Biggl(\prod_{k=\eta+2}^{\xi } T _{22}\bigl(t_{\rho(k)}\bigr)\otimes \prod_{l=2}^{\eta+1} T _{11}\bigl(t_{\rho(l)}\bigr)\Biggr).
\end{gather*}
In the first term, we move $T_{11}(t_1)$ through the product \smash{$\prod_{j=\eta+1}^{\xi } T_{12} \bigl(t_{\tau(j)}\bigr)$} using formula~\eqref{ABf}:
\begin{align*}
&T_{11} (t_1)\prod_{j=\eta+1}^{\xi } T_{12} \bigl(t_{\tau(j)}\bigr)=\Biggl( \prod_{j=\eta+1}^{\xi } \frac{t_1-t_{\tau(j)}-1}{t_1-t_{\tau(j)}} T_{12} \bigl(t_{\tau(j)}\bigr)\Biggr) T_{11} (t_1) \\[1mm]
&\qquad{}+\sum_{p=\eta+1}^\xi \frac{1}{t_1-t_{\tau(p)}}\Biggl( \prod_{\substack{j=\eta+1\\j\neq p}}^\xi \frac{t_{\tau(p)}-t_{\tau(j)}-1}{t_{\tau(p)}-t_{\tau(j)}} T_{12} \bigl(t_{\tau(j)}\bigr)\Biggr) T_{12} (t_1)T_{11} \bigl(t_{\tau(p)}\bigr).
\end{align*}
Similarly, in the second term we move $T_{22}(t_1)$ through the product $\prod_{m=2}^{\eta+1} T_{12} \bigl(t_{\rho(m)}\bigr)$ using formula~$\eqref{DBf}$:
\begin{align*}
T_{22} (t_1)\prod_{i=2}^{\eta+1 } T_{12} \bigl(t_{\rho(i)}\bigr)={}&\Biggl( \prod_{i=2}^{\eta+1}\frac{t_1-t_{\rho(i)}+1}{t_1-t_{\rho(i)}} T_{12} \bigl(t_{\rho(i)}\bigr)\Biggr) T_{22} (t_1)\\[1mm]
&{}{-}\sum_{s=2}^{\eta+1} \frac{1}{t_1-t_{\rho(s)}}\Biggl( \prod_{\substack{i=2\\i\neq s}}^{\eta+1} \frac{t_{\rho(s)}-t_{\rho(i)}+1}{t_{\rho(s)}-t_{\rho(i)}} T_{12} \bigl(t_{\rho(i)}\bigr) \Biggr) T_{12} (t_1)T_{22} \bigl(t_{\rho(s)}\bigr),
\end{align*}
After all, the right-hand side of~\eqref{RHS} becomes a sum of four terms:
\begin{equation*}
\Delta ( {\mathbb{B}}_\xi(\boldsymbol{t}) )=Y_1(\boldsymbol{t})+Y_2(\boldsymbol{t})+Y_3(\boldsymbol{t})+Y_4(\boldsymbol{t}),
\end{equation*}
where
\begin{gather*}
 Y_1(\boldsymbol{t})= \sum_{\eta=1}^{\xi} \frac{1}{(\xi-\eta) ! (\eta-1)!} \sum_{\tau\in S_{\xi-1}'} \Biggl[ \prod_{l=\eta+1}^{\xi } \frac{t_1-t_{\tau(l)}-1}{t_1-t_{\tau(l)}} \prod_{1< i\leq\eta<j \leq \xi} \frac{t_{\tau(i)}-t_{\tau(j)}-1}{t_{\tau(i)}-t_{\tau(j)}}\\[1mm]
 \hphantom{Y_1(\boldsymbol{t})= \sum_{\eta=1}^{\xi} \frac{1}{(\xi-\eta) ! (\eta-1)!} \sum_{\tau\in S_{\xi-1}'} \Biggl[}{}
 \times\Biggl(T_{12}(t_1) \prod_{i=2}^{\eta} T_{12} \bigl(t_{\tau(i)}\bigr) \otimes \prod_{j=\eta+1}^{\xi} T_{12} \bigl(t_{\tau(j)}\bigr) \Biggr)\\[1mm]
 \hphantom{Y_1(\boldsymbol{t})= \sum_{\eta=1}^{\xi} \frac{1}{(\xi-\eta) ! (\eta-1)!} \sum_{\tau\in S_{\xi-1}'} \Biggl[}{}
 \times \Biggl(\prod_{k=\eta+1}^{\xi } T _{22}\bigl(t_{\tau(k)}\bigr)\otimes T_{11}(t_1) \prod_{l=2}^{\eta} T _{11}\bigl(t_{\tau(l)}\bigr)\Biggr)\Biggr],
\\
Y_2(\boldsymbol{t})= \sum_{\eta=1}^{\xi-1} \frac{1}{(\xi-\eta) ! (\eta-1)!} \\
\hphantom{Y_2(\boldsymbol{t})= }{}
\times \sum_{\tau\in S_{\xi-1}'} \sum_{l=\eta+1}^\xi \Biggl[ \frac{1}{t_1-t_{\tau(l)}}\prod_{1< i\leq\eta<j \leq \xi} \frac{t_{\tau(i)}-t_{\tau(j)}-1}{t_{\tau(i)}-t_{\tau(j)}} \prod_{\substack{m=\eta+1\\m\neq l}}^\xi \frac{t_{\tau(l)}-t_{\tau(m)}-1}{t_{\tau(l)}-t_{\tau(m)}} \\
\hphantom{Y_2(\boldsymbol{t})=\times \sum_{\tau\in S_{\xi-1}'} \sum_{l=\eta+1}^\xi \Biggl[}{}
 \times\Biggl( T_{12} (t_1)\prod_{k=2}^{\eta} T_{12} \bigl(t_{\tau(k)}\bigr) \otimes T_{12}(t_1) \prod_{\substack{m=\eta+1\\m\neq l}}^\xi T_{12} \bigl(t_{\tau(m)}\bigr)\Biggr)\\
\hphantom{Y_2(\boldsymbol{t})=\times \sum_{\tau\in S_{\xi-1}'} \sum_{l=\eta+1}^\xi \Biggl[}{}
 \times\Biggl( \prod_{i=\eta+1}^{\xi } T _{22}\bigl(t_{\tau(i)}\bigr)\otimes T_{11}\bigl(t_{\tau(l)}\bigr)\prod_{j=2}^{\eta} T _{11}\bigl(t_{\tau(j)}\bigr)\Biggr) \Biggr].
\\
Y_3(\boldsymbol{t}) = \sum_{\eta=0}^{\xi-1} \frac{1}{(\xi-\eta-1) ! \eta!} \sum_{\rho\in S_{\xi-1}'} \Biggl[ \prod_{m=2}^{\eta+1 } \frac{t_{\rho(m)}-t_1-1}{t_{\rho(m)}-t_1} \prod_{1< i\leq\eta+1<j \leq \xi} \frac{t_{\rho(i)}-t_{\rho(j)}-1}{t_{\rho(i)}-t_{\rho(j)}} \\
\hphantom{Y_3(\boldsymbol{t}) = \sum_{\eta=0}^{\xi-1} \frac{1}{(\xi-\eta-1) ! \eta!} \sum_{\rho\in S_{\xi-1}'} \Biggl[}{}
 \times\Biggl( \prod_{m=2}^{\eta+1 } T_{12} \bigl(t_{\rho(m)}\bigr) \otimes T_{12}(t_1) \prod_{l=\eta+2}^{\xi } T_{12} \bigl(t_{\rho(l)}\bigr) \Biggr) \\
\hphantom{Y_3(\boldsymbol{t}) = \sum_{\eta=0}^{\xi-1} \frac{1}{(\xi-\eta-1) ! \eta!} \sum_{\rho\in S_{\xi-1}'} \Biggl[}{}
 \times\Biggl(T_{22}(t_1)\prod_{i=\eta+2}^{\xi } T _{22}\bigl(t_{\rho(i)}\bigr)\otimes \prod_{j=2}^{\eta+1} T _{11}\bigl(t_{\rho(j)}\bigr)\Biggr) \Biggr].
\\
Y_4(\boldsymbol{t})= -\sum_{\eta=1}^{\xi-1} \frac{1}{(\xi-\eta-1) ! \eta!} \\
\hphantom{Y_4(\boldsymbol{t})=}{}
 \times\sum_{\rho\in S_{\xi-1}'} \Biggl[ \sum_{k=2}^{\eta+1} \frac{1}{t_1-t_{\rho(k)}} \prod_{1< i\leq\eta+1<j \leq \xi} \frac{t_{\rho(i)}-t_{\rho(j)}-1}{t_{\rho(i)}-t_{\rho(j)}} \prod_{\substack{m=2\\m\neq k}}^{\eta+1} \frac{t_{\rho(k)}-t_{\rho(m)}+1}{t_{\rho(k)}-t_{\rho(m)}} \\
\hphantom{Y_4(\boldsymbol{t})=\times\sum_{\rho\in S_{\xi-1}'} \Biggl[}{}
 \times\Biggl(T_{12}(t_1) \prod_{\substack{m=2\\m\neq k}}^{\eta+1} T_{12} \bigl(t_{\rho(m)}\bigr) \otimes T_{12}(t_1) \prod_{l=\eta+2}^{\xi } T_{12} \bigl(t_{\rho(l)}\bigr) \Biggr) \\
\hphantom{Y_4(\boldsymbol{t})=\times\sum_{\rho\in S_{\xi-1}'} \Biggl[}{}
 \times\Biggl( T_{22}\bigl(t_{\rho(k)}\bigr)\prod_{i=\eta+2}^{\xi } T _{22}\bigl(t_{\rho(i)}\bigr)\otimes \prod_{j=2}^{\eta+1} T _{11}\bigl(t_{\rho(j)}\bigr)\Biggr) \Biggr].
\end{gather*}
To complete the proof, we will show that
\begin{equation}\label{Yt}
Y_1(\boldsymbol{t})+Y_3(\boldsymbol{t})=\sum_{\eta=0}^\xi \frac{1}{(\xi-\eta)!\eta!}\sum_{\rho\in S_\xi}P_{\eta,\xi-\eta}(\boldsymbol{t}^\rho) \qquad \text{and} \qquad Y_2(\boldsymbol{t})+Y_4(\boldsymbol{t})=0.
\end{equation}
We will start with the first equality in~\eqref{Yt}.

Observe that
\begin{equation*}
Y_1(\boldsymbol{t})= \sum_{\eta=1}^\xi \frac{1}{(\xi-\eta) ! (\eta-1)!} \sum_{\substack{\sigma\in S_\xi\\\sigma (1)=1}} P_{\eta,\xi-\eta}(\boldsymbol{t}^\sigma),
\end{equation*}
and
\begin{equation*}
Y_3(\boldsymbol{t})= \sum_{\eta=0}^{\xi-1} \frac{1}{(\xi-\eta-1) ! \eta !} \sum_{\substack{\sigma\in S_\xi\\\sigma (\eta+1)=1}}P_{\eta,\xi-\eta}(\boldsymbol{t}^\sigma).
\end{equation*}
On the other hand, we have
\begin{align*}
\sum_{\eta=0}^\xi \frac{1}{(\xi-\eta)!\eta!}\sum_{\sigma\in S_\xi} P_{\eta,\xi-\eta}(\boldsymbol{t}^\sigma)={}&\sum_{\eta=1}^\xi \frac{1}{(\xi-\eta)!\eta!} \sum_{\substack{\sigma\in S_\xi\\\sigma^{-1}(1)\in\{1,\dots,\eta\}}} P_{\eta,\xi-\eta}(\boldsymbol{t}^\sigma)\\
&{}{+} \sum_{\eta=0}^{\xi-1} \frac{1}{(\xi-\eta)!\eta!}\sum_{\substack{\sigma\in S_\xi\\\sigma^{-1}(1)\in\{\eta+1,\dots,\xi\}}}P_{\eta,\xi-\eta}(\boldsymbol{t}^\sigma).
\end{align*}
Denote by $s_{a,b}\in S_\xi$ the transposition of $a$ and $b$. Then we have
\begin{equation*}
\sum_{\substack{\sigma\in S_\xi\\\sigma^{-1}(1)\in\{1,\dots,\eta\}}} P_{\eta,\xi-\eta}(\boldsymbol{t}^\sigma) =\sum_{l=1}^\eta \sum_{\substack{\tau\in S_\xi,\\ \tau(1)=1}} P_{\eta,\xi-\eta}(\boldsymbol{t}^{\tau s_{1,l}})= \eta \sum_{\substack{\tau\in S_\xi,\\ \tau(1)=1}} P_{\eta,\xi-\eta}(\boldsymbol{t}^\tau).
\end{equation*}
For the first step, we used $l=\sigma^{-1}(1)$ and $\tau=\sigma s_{1,\sigma^{-1}(1)}$, so that $\tau(1)=1$. For the second step, we used the equality $P_{\eta,\xi-\eta}(\boldsymbol{t}^{\tau s_{1,l}})=P_{\eta,\xi-\eta}(\boldsymbol{t}^\tau)$.
Similarly,
\begin{equation*}
\sum_{\substack{\sigma\in S_\xi\\\sigma^{-1}(1)\in\{\eta+1,\dots,\xi\}}}P_{\eta,\xi-\eta}(\boldsymbol{t}^\sigma)=(\xi-\eta) \sum_{\substack{\rho\in S_\xi,\\ \rho(\eta+1)=1}} P_{\eta,\xi-\eta}(\boldsymbol{t}^\rho).
\end{equation*}
Therefore,
\begin{align*}
\sum_{\eta=0}^\xi \frac{1}{(\xi-\eta)!\eta!}\sum_{\sigma\in S_\xi} P_{\eta,\xi-\eta}(\boldsymbol{t}^\sigma)={}&\sum_{\eta=1}^\xi \frac{1}{(\xi-\eta)!(\eta-1)!} \sum_{\substack{\sigma\in S_\xi\\\sigma(1)=1}} P_{\eta,\xi-\eta}(\boldsymbol{t}^\sigma) \\
&{}{+} \sum_{\eta=0}^{\xi-1} \frac{1}{(\xi-1-\eta)!\eta!}\sum_{\substack{\sigma\in S_\xi\\\sigma(\eta+1)=1}}P_{\eta,\xi-\eta}(\boldsymbol{t}^\sigma)
\\
={}& Y_1(\boldsymbol{t})+Y_3(\boldsymbol{t}).
\end{align*}

Finally, we show that $Y_2(\boldsymbol{t})+Y_4(\boldsymbol{t})=0$.
Observe that $Y_2(\boldsymbol{t})$ can be written as
\begin{gather*}
Y_2(\boldsymbol{t})= \sum_{\eta=1}^{\xi-1} \frac{1}{(\xi-\eta) ! (\eta-1)!}\\
\hphantom{Y_2(\boldsymbol{t})=}{}
 \times\sum_{l=\eta+1}^\xi \sum_{\tau\in S_{\xi-1}'} \Biggl[ \frac{1}{t_1-t_{\tau(l)}} \prod_{1< i\leq\eta<j \leq \xi} \frac{t_{\tau(i)}-t_{\tau(j)}-1}{t_{\tau(i)}-t_{\tau(j)}} \prod_{\substack{m=\eta+1\\m\neq l}}^\xi \frac{t_{\tau(l)}-t_{\tau(m)}-1}{t_{\tau(l)}-t_{\tau(m)}}\\
\hphantom{Y_2(\boldsymbol{t})=\times\sum_{l=\eta+1}^\xi \sum_{\tau\in S_{\xi-1}'} \Biggl[}{}
\times\Biggl( T_{12} (t_1)\prod_{k=2}^{\eta} T_{12} \bigl(t_{\tau(k)}\bigr) \otimes T_{12}(t_1) \prod_{\substack{m=\eta+1\\m\neq l}}^\xi T_{12} \bigl(t_{\tau(m)}\bigr)\Biggr)\\
\hphantom{Y_2(\boldsymbol{t})=\times\sum_{l=\eta+1}^\xi \sum_{\tau\in S_{\xi-1}'} \Biggl[}{}
\times \Biggl( \prod_{i=\eta+1}^{\xi } T _{22}\bigl(t_{\tau(i)}\bigr)\otimes T_{11}(t_{\tau(l)})\prod_{j=2}^{\eta} T _{11}\bigl(t_{\tau(j)}\bigr)\Biggr) \Biggr].
\end{gather*}
Changing the summation variable in the inner sum, $\tau=\sigma s_{l,\eta+1}$, we obtain that
\begin{gather*}
\begin{split}
&Y_2(\boldsymbol{t})= \sum_{\eta=1}^{\xi-1} \frac{1}{(\xi-\eta) ! (\eta-1)!}\\
&\hphantom{Y_2(\boldsymbol{t})=}{}
 \times\sum_{l=\eta+1}^\xi \sum_{\sigma\in S_{\xi-1}'} \Biggl[ \frac{1}{t_1-t_{\sigma(\eta+1)}} \prod_{m=\eta+2}^\xi \frac{t_{\sigma(\eta+1)}-t_{\sigma(m)}-1}{t_{\sigma(\eta+1)}-t_{\sigma(m)}} \prod_{1< i\leq\eta<j \leq \xi} \frac{t_{\sigma(i)}-t_{\sigma(j)}-1}{t_{\sigma(i)}-t_{\sigma(j)}} \\
&\hphantom{Y_2(\boldsymbol{t})=\times\sum_{l=\eta+1}^\xi \sum_{\sigma\in S_{\xi-1}'} \Biggl[}{}
\times\Biggl( T_{12} (t_1)\prod_{k=2}^{\eta} T_{12} \bigl(t_{\sigma(k)}\bigr) \otimes T_{12}(t_1) \prod_{m=\eta+2}^\xi T_{12} \bigl(t_{\sigma(m)}\bigr)\Biggr)\\
&\hphantom{Y_2(\boldsymbol{t})=\times\sum_{l=\eta+1}^\xi \sum_{\sigma\in S_{\xi-1}'} \Biggl[}{}
 \times\Biggl( \prod_{i=\eta+1}^{\xi } T _{22}\bigl(t_{\sigma(i)}\bigr)\otimes \prod_{j=2}^{\eta+1} T _{11}\bigl(t_{\sigma(j)}\bigr)\Biggr) \Biggr].
 \end{split}
\end{gather*}
The expression under the inner sum over $\sigma$ does not depend on $l$, and after a redistribution of factors, we get
\begin{gather}
Y_2(\boldsymbol{t})= \sum_{\eta=1}^{\xi-1} \frac{1}{(\xi-\eta-1) ! (\eta-1)!} \nonumber\\
\hphantom{Y_2(\boldsymbol{t})=}{}
\times \sum_{\sigma\in S_{\xi-1}'}\Biggl[ \frac{1}{t_1-t_{\sigma(\eta+1)}}\prod_{m=\eta+2}^\xi \frac{t_{\sigma(\eta+1)}-t_{\sigma(m)}-1}{t_{\sigma(\eta+1)}-t_{\sigma(m)}}\prod_{i=2}^\eta \frac{t_{\sigma(i)}-t_{\sigma(\eta+1)}-1}{t_{\sigma(i)}-t_{\sigma(\eta+1)}} \nonumber\\
\hphantom{Y_2(\boldsymbol{t})=\times \sum_{\sigma\in S_{\xi-1}'}\Biggl[}{}
\times\prod_{1< i<\eta+1<j \leq \xi} \frac{t_{\sigma(i)}-t_{\sigma(j)}-1}{t_{\sigma(i)}-t_{\sigma(j)}}\nonumber\\
\hphantom{Y_2(\boldsymbol{t})=\times \sum_{\sigma\in S_{\xi-1}'}\Biggl[}{}
\times \Biggl( T_{12} (t_1)\prod_{k=2}^{\eta} T_{12} \bigl(t_{\sigma(k)}\bigr) \otimes T_{12}(t_1) \prod_{m=\eta+2}^\xi T_{12} \bigl(t_{\sigma(m)}\bigr)\Biggr) \nonumber\\
\hphantom{Y_2(\boldsymbol{t})=\times\sum_{\sigma\in S_{\xi-1}'}\Biggl[}{}
\times\Biggl( \prod_{i=\eta+1}^{\xi } T _{22}\bigl(t_{\sigma(i)}\bigr)\otimes \prod_{j=2}^{\eta+1} T _{11}\bigl(t_{\sigma(j)}\bigr)\Biggr) \Biggr].\label{final2}
\end{gather}
Similarly, $Y_4(\boldsymbol{t})$ can be written as
\begin{gather*}
Y_4(\boldsymbol{t})= \sum_{\eta=1}^{\xi-1} \frac{1}{(\xi-1-\eta) ! \eta!} \\
\hphantom{Y_4(\boldsymbol{t})= }{}
\times\sum_{k=2}^{\eta+1} \sum_{\rho\in S_{\xi-1}'} \Biggl[ \frac{1}{t_1-t_{\rho(k)}} \prod_{1< i\leq\eta+1<j \leq \xi} \frac{t_{\rho(i)}-t_{\rho(j)}-1}{t_{\rho(i)}-t_{\rho(j)}} \prod_{\substack{m=2\\m\neq k}}^{\eta+1} \frac{t_{\rho(k)}-t_{\rho(m)}+1}{t_{\rho(k)}-t_{\rho(m)}} \\
\hphantom{Y_4(\boldsymbol{t})= \times\sum_{k=2}^{\eta+1} \sum_{\rho\in S_{\xi-1}'} \Biggl[}{}
\times\Biggl(T_{12}(t_1) \prod_{\substack{m=2\\m\neq k}}^{\eta+1} T_{12} \bigl(t_{\rho(m)}\bigr) \otimes T_{12}(t_1) \prod_{l=\eta+2}^{\xi } T_{12} \bigl(t_{\rho(l)}\bigr) \Biggr) \\
\hphantom{Y_4(\boldsymbol{t})= \times\sum_{k=2}^{\eta+1} \sum_{\rho\in S_{\xi-1}'} \Biggl[}{}
\times\Biggl( T_{22}\bigl(t_{\rho(k)}\bigr)\prod_{i=\eta+2}^{\xi } T _{22}\bigl(t_{\rho(i)}\bigr)\otimes \prod_{j=2}^{\eta+1} T _{11}\bigl(t_{\rho(j)}\bigr)\Biggr) \Biggr],
\end{gather*}
and changing the summation variable in the inner sum, $\rho=\sigma s_{k,\eta+1}$, we get
\begin{gather*}
 Y_4(\boldsymbol{t})= -\sum_{\eta=1}^{\xi-1} \frac{1}{(\xi-1-\eta) ! \eta!}\\
 \hphantom{Y_4(\boldsymbol{t})=}{}
 \times\sum_{k=2}^{\eta+1}\sum_{\sigma\in S_{\xi-1}'} \Biggl[ \frac{1}{t_1-t_{\sigma(\eta+1)}} \prod_{m=2}^{\eta} \frac{t_{\sigma(\eta+1)}-t_{\sigma(m)}+1}{t_{\sigma(\eta+1)}-t_{\sigma(m)}}\prod_{1< i\leq\eta+1<j \leq \xi} \frac{t_{\sigma(i)}-t_{\sigma(j)}-1}{t_{\sigma(i)}-t_{\sigma(j)}} \\
 \hphantom{Y_4(\boldsymbol{t})=\times\sum_{k=2}^{\eta+1}\sum_{\sigma\in S_{\xi-1}'} \Biggl[}{}
\times\Biggl(T_{12}(t_1) \prod_{m=2}^{\eta} T_{12} \bigl(t_{\sigma(m)}\bigr) \otimes T_{12}(t_1) \prod_{l=\eta+2}^{\xi } T_{12} \bigl(t_{\sigma(l)}\bigr) \Biggr)\\
 \hphantom{Y_4(\boldsymbol{t})=\times\sum_{k=2}^{\eta+1}\sum_{\sigma\in S_{\xi-1}'} \Biggl[}{}
\times \Biggl( \prod_{i=\eta+1}^{\xi } T _{22}\bigl(t_{\sigma(i)}\bigr)\otimes \prod_{j=2}^{\eta+1} T _{11}\bigl(t_{\sigma(j)}\bigr)\Biggr) \Biggr].
\end{gather*}
The expression under the inner sum over $\sigma$ does not depend on $k$, and after a redistribution of factors, we get
\begin{gather}
 Y_4(\boldsymbol{t})= -\sum_{\eta=1}^{\xi-1} \frac{1}{(\xi-\eta-1)!(\eta-1)!}\nonumber\\
 \hphantom{Y_4(\boldsymbol{t})=}{}
 \times\sum_{\sigma\in S_{\xi-1}'} \Biggl[ \frac{1}{t_1-t_{\sigma(\eta+1)}}\prod_{m=2}^{\eta} \frac{t_{\sigma(\eta+1)}-t_{\sigma(m)}+1}{t_{\sigma(\eta+1)}-t_{\sigma(m)}}\prod_{q=\eta+2}^\xi\frac{t_{\sigma(\eta+1)}-t_{\sigma(q)}-1}{t_{\sigma(\eta+1)}-t_{\sigma(q)}}\nonumber\\
 \hphantom{Y_4(\boldsymbol{t})=\times\sum_{\sigma\in S_{\xi-1}'} \Biggl[}{}
 \times\prod_{1< i<\eta+1<j \leq \xi} \frac{t_{\sigma(i)}-t_{\sigma(j)}-1}{t_{\sigma(i)}-t_{\sigma(j)}} \nonumber\\
 \hphantom{Y_4(\boldsymbol{t})=\times\sum_{\sigma\in S_{\xi-1}'} \Biggl[}{}
\times\Biggl(T_{12}(t_1) \prod_{m=2}^{\eta} T_{12} \bigl(t_{\sigma(m)}\bigr) \otimes T_{12}(t_1) \prod_{l=\eta+2}^{\xi } T_{12} \bigl(t_{\sigma(l)}\bigr) \Biggr)\nonumber \\
 \hphantom{Y_4(\boldsymbol{t})=\times\sum_{\sigma\in S_{\xi-1}'} \Biggl[}{}
\times\Biggl( \prod_{i=\eta+1}^{\xi } T _{22}\bigl(t_{\sigma(i)}\bigr)\otimes \prod_{j=2}^{\eta+1} T _{11}\bigl(t_{\sigma(j)}\bigr)\Biggr) \Biggr].\label{final4}
\end{gather}
Formulae~\eqref{final2} and~\eqref{final4} show that $Y_2(\boldsymbol{t})+Y_4(\boldsymbol{t})=0$. This completes the proof of formula~\eqref{state}. Proposition~\ref{twotensorsA} is proved.
\end{proof}

\subsection*{Acknowledgements}
The authors thank the referees for very careful reading of this paper and their valuable
suggestions. The second author is supported in part by Simons Foundation grants 430235, 852996.

\pdfbookmark[1]{References}{ref}
\LastPageEnding

\end{document}